\documentclass{siamltex}
\usepackage{amsmath}
\usepackage{amsfonts}
\usepackage{amssymb}
\usepackage{graphicx}
\usepackage{color}
\usepackage{bm}
\usepackage{subcaption}

\allowdisplaybreaks

\newtheorem{remark}[theorem]{Remark}
\newtheorem{assumption}[theorem]{Assumption}

\newcommand{\be}{\begin{equation}}
\newcommand{\ee}{\end{equation}}
\newcommand{\bea}{\begin{eqnarray}}
\newcommand{\eea}{\end{eqnarray}}
\newcommand{\beas}{\begin{eqnarray*}}
\newcommand{\eeas}{\end{eqnarray*}}

\newcommand{\Grad}{\ensuremath{\nabla}}

\newcommand{\dt}{\Delta{t}}
\newcommand{\cb}{C_{b^{\ast}}}
\newcommand{\mhv}{\mathcal{H}_{V}}
\newcommand{\mhp}{\mathcal{H}_{P}}
\newcommand{\vvvert}{|||}
\newcommand{\alb}{(1 + \alpha C^{H^{1}}_{r})}
\newcommand{\bstar}[3]{b^*\left(#1,#2,#3\right)}

\input{tcilatex}

\usepackage{url}

\begin{document}

\title{Error Analysis of Supremizer Pressure Recovery for POD based Reduced Order Models of the time-dependent Navier-Stokes Equations}

\author{Kiera Kean \and Michael Schneier}
%
%
\maketitle
%
%
%

\begin{abstract}
	For incompressible flow models, the pressure term serves as a Lagrange multiplier to ensure that the incompressibility constraint is satisfied. In engineering applications, the pressure term is necessary for calculating important quantities based on stresses like the lift and drag. For reduced order models generated via a Proper orthogonal decomposition, it is common for the pressure to drop out of the equations and produce a velocity-only reduced order model. To recover the pressure, many techniques have been numerically studied in the literature; however, these techniques have undergone little rigorous analysis. In this work, we examine two of the most popular approaches: pressure recovery through the Pressure Poisson equation and recovery via the momentum equation through the use of a supremizer stabilized velocity basis. We examine the challenges that each approach faces and prove stability and convergence results for the supremizer stabilized approach. We also investigate numerically the stability and convergence of the supremizer based approach, in addition to its performance against the Pressure Poisson method.
\end{abstract}

\section{Introduction}
Let $\Omega \subset \mathbb{R}^{d}$, $d=2,3$ be a regular open domain with Lipschitz continuous boundary $\Gamma$. We consider the Navier-Stokes equations (NSE) with no-slip boundary conditions:
\begin{equation}\label{eqn:nse-1}
\begin{aligned}
&u_t + u\cdot\nabla u + \nabla p - \nu\Delta u =  f,\ \text{and } \nabla \cdot u =  0,\ \text{in} \ \Omega \times (0,T]   \\
&u   =  0, \ \text{on} \ \Gamma \times (0,T], \ \text{and } u(x,0)  =  u_0(x), \ \text{in} \ \Omega,  \\
\end{aligned}
\end{equation}
where $u$ is the velocity, $p$ is the pressure, $f$ is the known body force, and $\nu$ is the viscosity. 

In recent years, there has been a growing interest in the application of reduced order models (ROMs) to modeling incompressible flows \cite{FMPT18,HRS15,NMT11,RAMBR17,SR18,XWWI17,V05}. Galerkin-based ROMs use experimental data, or solutions generated from full-order numerical schemes, i.e., finite element or finite volumes schemes, to generate a low dimensional basis. Due to the low dimensionality of the ROM basis, computational costs can be orders of magnitude smaller when compared to these full-order schemes. In practice, the data used to generate the ROM basis will often be weakly divergence-free. This divergence-free property causes the pressure term to drop out of the ROM formulation, leading to a velocity-only ROM. However, in almost every setting, accurate recovery of the pressure is required to calculate forces on walls or immersed boundaries. Additionally, the pressure term can be used to calibrate codes and models with (reliable) pressure data. 

The problem tackled herein is how to recover the discrete pressure, $p_{m}$, reliably and accurately from a (discretely) divergence-free POD velocity $u_{r}$. Several approaches have been used in the literature, but no validation of their accuracy and stability has been conducted. The two most popular approaches are: 
\\\\
	{ (1) Solving the Pressure Poisson equation (PPE)}: \ 
\begin{equation}
\Delta p_{m}
=  - \nabla\cdot((u_{r} \cdot \nabla) u_{r})  + \nabla \cdot f + BC
\quad \mbox{in } \Omega\,,
\label{eqn:pressure-poisson}
\end{equation}
which is obtained by taking the divergence of the NSE~\eqref{eqn:nse-1}. Here, $BC$ is a Neumann boundary condition which will be derived in Section \ref{sec:PPE}.
\\\\
{(2) Determining the pressure via the momentum equation recovery formulation  \newline (MER):} \
\begin{equation}
\nabla p_{m} = u_{t,r} + u_{r} \cdot \nabla u_{r} - \nu \Delta u_{r} + f \quad \mbox{in } \Omega\,.
\end{equation}
 In practice, this involves using the supremizer stabilization technique developed in \cite{BMQR15,RV07} to ensure compatibility between the pressure and velocity spaces.
 
 Herein, we analyze the stability and convergence  of the MER method's, and briefly review the PPE approach. In the ROM literature, the PPE has yielded accurate results; however, we we will see in the derivation of the discrete equations, as well as in the numerical experiments that the Neumann boundary condition leads to a loss of accuracy, especially within the boundary layer. 
 
 The MER method does not require any boundary conditions. Surprisingly, however, it does not work universally. Its reliability will be dependent on the classic inf-sup condition, as well as an \textit{a priori} computable constant dependent on the angle between the initial POD velocity space and supremizer space. We show in the numerical experiments that for the same number of basis functions, the MER approach yields more accurate solutions for the pressure than the PPE method.
 
 The rest of this paper is organized as follows: In Section \ref{sec:notation}, we introduce notation and state preliminary results. In Section \ref{sec:POD_sec}, we outline the construction of our ROM via a proper orthogonal decomposition. In Section \ref{sec:pressure_recovery}, we present the derivation of the PPE and MER. In Section \ref{sec:error_analysis}, we prove stability and convergence results for the PPE and MER formulations. In Section \ref{sec:numerical_experiments}, we numerically investigate the performance of these pressure recovery techniques. In Section \ref{sec:conclusions}, we end the paper with conclusions and discussion of future research directions.

\subsection{Related Work}
 
For pressure recovery, the PPE has been studied extensively within both the finite element setting \cite{GM87,JL04,SSPG06} and the ROM setting \cite{ANR09,CIJS14,NPM05,SHMLR17}. In \cite{CIJS14}, a numerical comparison was performed for a formulation of the PPE involving pressure basis functions versus one which strictly relied on the velocity modes. In \cite{NPM05}, the authors explored the need for a pressure term, determined via the PPE, for ROM simulations of shear flows. In \cite{SHMLR17}, the authors used the PPE to recover the pressure for a finite volume based ROM of vortex shedding around a circular cylinder.

  The supremizer stabilization approach for recovering the pressure was introduced in \cite{BMQR15} for the parameterized steady NSE. It was extended to the case where a strongly divergence-free POD velocity basis is used in \cite{FBKR19}. Supremizers have also been used in the context of Petrov Galkerin methods in \cite{AB15, CC19, Y14}.

A different class of approaches studied for recovering the pressure incorporates a pressure stabilization. This approach relaxes the incompressbility constraint, ensuring that the pressure term does not drop out of the ROM formulation. These include the artificial compression scheme studied in the ROM setting in \cite{DILMS19} and the Local Projection Stabilization ROM studied in \cite{R19}.  
\section{Notation and Preliminaries}\label{sec:notation}

In this section, we establish notation and collect preliminary results needed for the numerical analysis and experiments in the following sections.  We denote by $\|\cdot\| = \|\cdot\|_{0}$  the $L^{2}(\Omega)$ norm and by $(\cdot,\cdot)$ the $L^{2}(\Omega)$ inner product. The standard velocity space $X$ and pressure space $Q$ are defined as:
$$
\begin{aligned}
X : =& H^{1}_{0}(\Omega)^{d} = \{ v \in H^{1}(\Omega)^{d} \,:\, v|_{\Gamma} = 0 \} \\
Q : =& L^{2}_{0}(\Omega) = \{ q \in L^{2}(\Omega) \,:\, \int_{\Omega} q dx = 0 \}.
\end{aligned}
$$
For functions $v \in X$, the Poincar\'{e} inequality holds
\begin{equation*}
\begin{aligned}
&\|v\| \leq C_{P}\|\nabla v\|.
\end{aligned}
\end{equation*}
\noindent The space $H^{-1}(\Omega)$ denotes the dual space of bounded linear functionals defined on $H^{1}_{0}(\Omega)=\{v\in H^{1}(\Omega)\,:\,v=0 \mbox{ on } \Gamma\}$; this space is equipped with the norm
$$
\|f\|_{-1}=\sup_{0\neq v\in X}\frac{(f,v)}{\| \nabla v\| } 
\quad\forall f\in H^{-1}(\Omega).
$$

\noindent We assume that the solution of the NSE is a strong solution satisfying the weak formulation
\begin{equation}\label{wfwf}
\begin{aligned}
(u_{t},v)+(u\cdot\nabla u,v)+\nu(\nabla u,\nabla v)-(p
,\nabla\cdot v)  &  =(f,v)&\quad\forall v\in X\\
(\nabla\cdot u,q)  &  =0&\quad\forall q\in Q.\\
\end{aligned}
\end{equation} 
We will consider a discretization of the time interval $[0,T]$ into $N$ separate intervals such that $\Delta t = \frac{T}{N}$ and $t_{n} = n \Delta t$ for $n = 0, \ldots, N$.
We then define the norms
$$
||v||_{p,s} : = \Big(\int_{0}^{T}\|v(\cdot,t)\|_{s}^{p}dt\Big)^{\frac{1}{p}}
\qquad \text{and} \qquad 
||v||_{\infty,s} := \text{ess\,sup}_{[0,T]}\|v(\cdot,t)\|_{s} ,
$$
and their  discrete counterparts
$$
|||v|||_{p,s} : = \Big(\sum_{n=0}^{N}\|v^{n}\|_{s}^{p} \Delta t\Big)^{\frac{1}{p}}
\qquad \text{and} \qquad 
|||v|||_{\infty,s} := \max_{0 \leq n \leq N}\| v^{n}\|_{s}.
$$


For the spatial discretization of the NSE, we use a conforming finite element space for the velocity $X_{h}\subset X$ and pressure $Q_{h}\subset Q$ based on a regular triangulation of $\Omega$ having maximum triangle diameter $h$.  
We assume that the finite element spaces satisfy the discrete inf-sup condition: There exists a constant $\beta_{h} >0$ independent of $h$ such that
\begin{equation}\label{eqn:inf-supFE}
\inf_{q_{h} \in Q_{h} \backslash \{0\} }  \sup_{v_{h} \in X_{h} \backslash \{0\}} \frac{(\nabla \cdot v_{h},q_{h})}{\| \nabla v_{h} \| \|q_{h}\|} \geq \beta_{h}.
\end{equation}
In addition, we assume that these finite element spaces fulfill the following approximation properties: 
$$
\begin{aligned}
\inf_{v_h\in X_h}\| v- v_h \|&\leq C(v,\nu) h^{s+1}&\forall v\in H^{s+1}(\Omega)^d,\\
\inf_{v_h\in X_h}\| \nabla ( v- v_h )\|&\leq C(v,\nu) h^s&\forall v\in H^{s+1}(\Omega)^d,\\
\inf_{q_h\in Q_h}\|  q- q_h \|&\leq C(q,\nu) h^k&\forall q\in H^{k}(\Omega).
\end{aligned}
$$

\noindent We define the trilinear form
$$
b(w,u,v) = (w\cdot\nabla u,v) 
\qquad\forall u,v,w\in H^1(\Omega)^d
$$
and the explicitly skew-symmetric trilinear form  by 
$$
b^{\ast}(w,u,v):=\frac{1}{2}(w\cdot\nabla u,v)-\frac{1}{2}(w\cdot\nabla v,u)
\qquad\forall u,v,w\in H^1(\Omega)^d \, .
$$
The term $b^{\ast}$ satisfies the following bound
\begin{lemma}\label{lemma:trilinear}
There exists a constant $C_{b^{\ast}}>0$ only dependent on the domain $\Omega$ such that 
\begin{equation*}
b^{\ast}(w,u,v)\leq C_{b^{\ast}} \|\nabla w\|  \| \nabla u\|  \| \nabla
v \| \qquad\forall u, v, w \in X.
\end{equation*}
\end{lemma}
\begin{proof}
	See Lemma 6.11 of \cite{J16}.
\end{proof}
 We define the space of discretely divergence free functions as
\begin{equation}
V^{div}_{h}:= \{v_{h} \in X_{h} : (\nabla \cdot v_{h},q_{h}) = 0 \ \forall q_{h}\in Q_{h}\} \subset X.
\end{equation}
From Hilbert space theory, the function space $X_{h}$ can be decomposed into the orthogonal subspaces
\begin{equation}\label{ortho_subspaces}
X_{h} = V^{div}_{h} \oplus (V^{div}_{h})^{\perp},
\end{equation}
where the orthogonality is in the sense of the $H^{1}$ inner product.

Throughout the rest of this paper we assume that the solution to the NSE satisfies the following regularity assumptions:
\\
\begin{assumption}\label{assumption:regularity} In \eqref{wfwf} we assume that $u$, $p,$ and $f$ satisfy:
	\begin{equation*}
	\begin{aligned}
	&u \in L^{\infty}(0,T,X \cap H^{s+1}(\Omega)),  u_{t} \in L^{2}(0,T,H^{s+1}(\Omega)), u_{tt} \in L^{2}(0,T,H^{s+1}(\Omega)),
	\\
	&f \in L^{2}(0,T,L^{2}(\Omega)), p \in L^{2}(0,T,Q \cap H^{k}(\Omega)).
	\end{aligned}
	\end{equation*}
\end{assumption}

The calculation of snapshots to construct the ROM in the ensuing sections is done using the $P^{2}-P^{1}$ Taylor-Hood finite element pair along with a backward Euler time discretization. Specifically, given $u^{0}_h, \in X_h$  for $n=1,2,\ldots,N-1$, we find $u^{n+1}_h\in X_h$ and $p_h^{n+1}\in Q_h$ satisfying

\begin{equation}\label{eqn:BE_FEM}
\begin{aligned}
&\Big(\frac{u^{n+1}_h - u^{n}_h}{\Delta t}, v_h \Big) + b^{\ast}(u_h^{n} , u^{n+1}_h ,v_h) +\nu (\nabla u^{n+1}_h, \nabla v_h) \\
&- (p^{n+1}_h , \nabla \cdot v_h)  =( f^{n+1}, v_h) \quad \quad \qquad \forall v_h\in X_h\\
&(\nabla \cdot u_h^{n+1}, q_h )= 0 \qquad \qquad  \qquad \qquad \ \ \ \forall  q_h\in Q_h.
\end{aligned}
\end{equation}

It has been shown in Theorem 7.78 of \cite{J16}, using Taylor-Hood elements and under the regularity conditions given in Assumption \ref{assumption:regularity}, \eqref{eqn:BE_FEM} will satisfy the following error estimate
\begin{equation}
\|u(t^{N}) - u_{h}^{N}\|^{2} + \nu \Delta t \sum_{n=1}^{N} \|\nabla u(t^{n}) - u_{h}^{n} \|^{2} \leq C(\nu) \left(h^{2s}(1 + \nu^{-1}\|p\|_{\infty,k}^{2}) + \Delta t^{2} \right),
\end{equation}
with $C$ independent of $h$, $p$, and $\Delta t$.


\section{Proper Orthogonal Decomposition Preliminaries}
\label{sec:POD_sec}
In this section, we briefly describe the POD method. We will closely follow the notation  and presentation in \cite{DILMS19}. A more detailed description of this method can be found in \cite{KV01}.

We discretize the time interval $[0,T]$ into $N$ separate intervals such that $\Delta t = \frac{T}{N}$ and $t_{n} = n \Delta t$ for $n = 0, \ldots, N$.  We will denote by $u_{h}^{n}(x)\in X_h$, $p_{h}^{n}(x)\in Q_h$, $n=0,\ldots,N$, the finite element solution to \eqref{eqn:BE_FEM} evaluated at $t=t_n$, $n=1,\ldots,N$. 

Letting ${u}_S^{n}$ and ${p}_S^{n}$ be the vector of coefficients corresponding to the finite element functions $u_{h}^{n}(x)$ and $p_{h}^{n}(x)$, we define the velocity snapshot matrix $\mathbb{V}$ and pressure snapshot matrix $\mathbb{P}$ as
\begin{equation*}
\begin{aligned}
\mathbb{V} = \big({u}_S^{0},{u}_S^{1}, \ldots , {u}_S^{N})\ \ \text{and} \ \ \mathbb{P} = \big({p}_S^{0},{p}_S^{1}, \ldots , {p}_S^{N}).
\end{aligned}
\end{equation*}
%
We consider the set of finite element velocity $\{u^{n}_{h,S}\}_{n=0}^{N}$ and pressure $\{p^{n}_{h,S}\}_{n=0}^{N}$ functions   corresponding to the velocity and pressure snapshots.
Defining the velocity and pressure spaces spanned by these functions as
\begin{equation*}
X_{h,S} :=\text{span}\{u^{n}_{h,S}\}_{n=0}^{N}  \subset X_h \ \ \text{and} \ \
Q_{h,S} :=\text{span}\{p^{n}_{h,S}\}_{n=0}^{N}  \subset Q_h,
\end{equation*}
the POD method then seeks a low-dimensional representation of these spaces. Denoting by $\{{\varphi_i(x)}\}_{i=1}^r$ the velocity POD basis and $\{{\psi_i(x)}\}_{i=1}^m$ the pressure POD basis we define the reduced velocity and pressure spaces as
\begin{equation*}
X_r :=\text{span}\{{\varphi}_i\}_{i=1}^r \subset X_{h,S} \subset X_h \ \ \text{and} \ \
Q_m :=\text{span}\{{\psi}_i\}_{i=1}^m  \subset P_{h,S} \subset Q_h.
\end{equation*}
We let $\delta_{ij}$ denote the Kronecker delta and $\mathcal{H}_{V}$ and $\mathcal{H}_{P}$ a Hilbert space for the velocity and pressure space, respectively. The POD method determines these bases by solving the constrained minimization problems: find $\{{\varphi}_i\}_{i=1}^r$ and $\{{\psi}_i\}_{i=1}^m$ satisfying
\begin{equation}\label{Min-velocity}
\begin{aligned}
\frac{1}{N+1}\min  \sum_{n=0}^{N} \Big \|  u_{h}^{n}-\sum_{j=1}^r (u_{h}^{n}, \varphi_j)_{\mathcal{H}_{V}}\varphi_j\Big \|_{\mathcal{H}_{V}}^2 \\
\text{subject to } (\varphi_i, \varphi_j)_{\mathcal{H}_{V}}= \delta_{ij}\quad\mbox{for $i,j=1,\ldots,r$},
\end{aligned}
\end{equation}
and 
\begin{equation}\label{Min-pressure}
\begin{aligned}
\frac{1}{N+1} \min  \sum_{n=0}^{N_{}} \Big \|  p_{h}^{n}-\sum_{j=1}^m (p_{h}^{n}, \psi_j)_{\mathcal{H_{P}}}\psi_j\Big \|_{\mathcal{H}_{P}} ^2 \\
\text{subject to } (\psi_i, \psi_j)_{\mathcal{H}_{P}}= \delta_{ij}\quad\mbox{for $i,j=1,\ldots,m$}.
\end{aligned}
\end{equation}
Defining the velocity and pressure correlation matrices $\mathbb{C}_{V} = \frac{1}{N+1}(u^{n}_{S},u^{k}_{S})_{\mathcal{H}_{V}}$ and $\mathbb{C}_{P} = \frac{1}{N+1}(p^{n}_{S},p^{k}_{S})_{\mathcal{H}_{P}}$ for $n,k = 0, \ldots N,$ these problems can then be solved by considering the eigenvalue problems
\begin{equation*}
\mathbb{C}_{V}\vec{a}_{i} = \lambda_{i}\vec{a}_{i},
\end{equation*}
and
\begin{equation*}
\mathbb{C}_{P}\vec{b}_{i} = \sigma_{i}\vec{b}_{i}.
\end{equation*}
The eigenvalues for ${C}_{V}$, $\lambda_{1} \geq \lambda_{N_V} > 0$, and ${C}_{P}$, $\sigma_{1} \geq \sigma_{N_P}  > 0$, are sorted in descending order. Here, $N_{V}$ and $N_{P}$ are the rank of $\mathbb{V}$ and $\mathbb{P},$ respectively. It follows that the finite element basis coefficients corresponding to the POD basis functions will be given by
\begin{equation*}
\vec\varphi_i = \frac{1}{\sqrt{\lambda_i}}\mathbb{C}_{V}\vec{a}_{i}, \ \ \ i = 1, \ldots, r,
\end{equation*}
and
\begin{equation*}
\vec\psi_i = \frac{1}{\sqrt{\sigma_i}}\mathbb{C}_{P}\vec{b}_{i}, \ \ \ i = 1, \ldots, m.
\end{equation*}

Throughout the rest of this paper, we will assume that $\mhv = L^{2}$ and $\mhp = L^{2}$.  POD error analysis has been conducted for $\mhv = H_{0}^{1}$ in the semidiscrete setting for the NSE in \cite{KV02}. Analysis and numerical tests comparing the different POD bases was conducted in the semidiscrete setting for the heat equation using a variety of different error norms in \cite{IW142} and for the NSE in \cite{S14}.  We note that results in this paper could be extended to the case where  $\mhv = H_{0}^{1}$ and $\mhp = H^{1}$, but do not do so here for clarity of presentation. A rigorous comparison between the $L^{2}$ and $H^{1}$ POD basis in the fully discrete setting for the velocity approximation and the pressure recovery techniques explored in this paper is a subject of ongoing research.

Using the velocity POD basis $\{\varphi\}_{i=1}^{r}$ we will construct the {BE-ROM} scheme. We seek a solution in $X_{r}$ using the POD basis $\{\varphi_{i}\}_{i=1}^{r}$ as opposed to the finite element basis as done in \eqref{eqn:BE_FEM}. The {BE-ROM} scheme can be written as:
\begin{equation}\label{eqn:BE_ROM}
\Big(\frac{u^{n+1}_r - u_{r}^{n}}{\Delta t}, \varphi \Big) + b^{\ast}(u_r^{n} , u^{n+1}_r ,\varphi) +\nu (\nabla u^{n+1}_r, \nabla \varphi)  =( f^{n+1}, \varphi) \ \  \forall \varphi \in X_r.
\end{equation}
The terms involving the pressure have dropped out of \eqref{eqn:BE_ROM} due to the fact that $X_{r} \subset V^{div}_{h}$, yielding a velocity only ROM.

\section{Pressure Recovery Formulations}\label{sec:pressure_recovery}
It was show in the derivation of {BE-ROM} \eqref{eqn:BE_ROM}, due to the fact that $X_{r} \subset V^{div}_{h}$, the pressure term drops out of the formulation yielding a velocity-only ROM. In this section, we review two ways in which the pressure can be recovered from the velocity solution $u_{r}^{n+1}$.

\subsection{Momentum Equation Recovery}
The MER approach for recovering the pressure involves just the weak form of the momentum equation, i.e., given the ROM solution $u^{n}_{r}$, $u_{r}^{n+1},$ determined by \eqref{eqn:BE_ROM}, find $p^{n+1}_{m} \in Q_{m}$ satisfying
\begin{equation}\label{eqn:mom-temp}
\begin{aligned}
(p^{n+1}_m , \nabla \cdot s)  &= -( f^{n+1}, s) +\Big(\frac{u^{n+1}_r - u^{n}_r}{\Delta t}, s \Big) + b^{\ast}(u_r^{n} , u^{n+1}_r ,s)
\\
  &+\nu (\nabla u^{n+1}_r, \nabla s) \ \ \ \quad \quad \qquad \qquad \forall s\in S \subset  (V^{div}_{h})^{\perp}.\\
\end{aligned}
\end{equation}
 This method was studied in the ROM setting for the steady NSE in \cite{FBKR19}. An important consideration is that the test space $S$ must be determined such that it is inf-sup stable with respect to the pressure space $Q_{m}$. To do so, we follow the same approach from \cite{FBKR19}, and use the supremizer stabilization method developed in \cite{BMQR15,RV07}.
 \\
\begin{remark}
	Due to the fact that $u^{n+1}_{r} \in V_{h}^{div}$ and $s\in (V^{div}_{h})^{\perp}$, we have $\nu (\nabla u^{n+1}_r, \nabla s) = 0$ in \eqref{eqn:mom-temp}.
\end{remark}

\subsubsection{Supremizer Stabilization and weak formulation}
 We consider the discrete inf-sup condition \eqref{eqn:inf-supFE} replacing the pressure finite element space with the ROM space $Q_{m}$  
\begin{equation}\label{eqn:inf-supROM1}
\inf_{\psi \in Q_{m} \backslash \{0\} }  \sup_{v_{h} \in X_{h} \backslash \{0\}} \frac{(\nabla \cdot v_{h},\psi)}{\| \nabla v_{h} \| \|\psi\|}.
\end{equation}
Given a function $p_{m} \in Q_{m}$, its supremizer will be the velocity function $s_{h} \in X_{h}$ that realizes the inf-sup condition in \eqref{eqn:inf-supROM1}. This can be interpreted as the Reisz representation in $X_{h}$ of the linear functional $(\nabla \cdot, p_{m})$, i.e., the solution of find $s_{h} \in X_{h}$ such that
\begin{equation}\label{eqn:supremizer}
(\nabla s_{h}, \nabla v_{h}) =  - (\nabla \cdot v_{h},p_{m}) \ \ \ \forall v_{h} \in X_{h}.
\end{equation}
The supremizer enrichment algorithm consists of solving \eqref{eqn:supremizer} for each basis function $\{\psi_{i}\}_{i=1}^{m}$. Then, applying a Gram-Schmidt orthonormalization procedure to the set of solutions yields a set of basis functions $\{\zeta_{i}\}_{i=1}^{m}$. Letting
\begin{equation}
S_m :=\text{span}\{{\zeta}_i\}_{i=1}^m  \subset (V^{div}_{h})^{\perp} \subset X_h , 
\end{equation}
the following inf-sup stability condition holds for the spaces $S_{m}$ and $Q_{m}$.

\begin{lemma}\label{lemma:inf-sup}
	Let $\beta_{h} >0$ be the inf-sup constant for the finite element basis in \eqref{eqn:inf-supFE}. The spaces $S_{m}$ and $Q_{m}$ will then be inf-sup stable with a constant $\beta_{m} \geq \beta_{h}$, i.e.,
	\begin{equation}\label{eqn:inf-supROM2}
	\beta_{m} = \inf_{\psi \in Q_{m} \backslash \{0\} }  \sup_{\zeta \in S_{m} \backslash \{0\}} \frac{(\nabla \cdot \zeta ,\psi)}{\| \nabla \zeta \| \|\psi\|} \geq \beta_{h}.
	\end{equation}
\end{lemma}
\begin{proof}
	See section 4 of \cite{BMQR15}.
\end{proof}

Using the space $S_{m}$ in \eqref{eqn:mom-temp} the MER formulation is then given by: find $p_{m}^{n+1} \in Q_{m}$ satisfying
\begin{equation}\label{eqn:posteriori-momentum}
\begin{aligned}
(p_{m}^{n+1}, \nabla \cdot \zeta) =  &\Big(\frac{u^{n+1}_r - u^{n}_r}{\Delta t}, \zeta \Big) + b^{\ast}(u_r^{n} , u^{n+1}_r,\zeta) 
\\
& - ( f^{n+1}, \zeta)  \ \  \forall \zeta \in S_m.
\end{aligned}
\end{equation} 
It can be shown (see section 4 of \cite{FBKR19}) that solving \eqref{eqn:BE_ROM} followed by \eqref{eqn:posteriori-momentum} is equivalent to the coupled system generated by discretizing \eqref{eqn:BE_FEM} with the combined velocity basis $X_{r} \bigoplus S_{m}$ and pressure space $Q_{m}$. The disadvantage to this approach is that it results in needing to solve a system of size $r + 2m$ instead of separate ones of size $r$ and $m$.
\\
\begin{remark}
	We note that the computational cost of this approach is comparable to other methods used for pressure recovery in the time-dependent setting. In \cite{BMQR15}, the authors considered the steady NSE in a parameterized domain.  This resulted in the inf-sup constant \eqref{eqn:inf-supROM2} to be parameter dependent. Therefore, each time a different parameter was sampled, the supremizer stabilization algorithm needed to be rerun for \eqref{eqn:inf-supROM2} to be satisfied. Because of the large computational cost, the authors proposed an approximate supremizer algorithm that did not rigorously satisfy \eqref{eqn:inf-supROM2}. We stress for the problem setting studied in this paper the inf-sup constant will not be parameter dependent. Therefore the supremizer stabilization algorithm only needs to be run once in the offline stage. This cost is negligible compared to the cost of generating the snapshot matrices $\mathbb{V}$  and $\mathbb{P}$ in the offline phase.
\end{remark}

\subsection{Pressure Poisson}\label{sec:PPE}

In the ROM literature, the most frequently used technique for recovering the pressure is the PPE. The PPE has been studied in the continuous, finite difference, and finite element settings \cite{GM87,JL04,SSPG06}. In the ROM setting, numerical studies have been performed in \cite{CIJS14,NPM05}. 
In this section, we rederive the PPE and its corresponding weak formulation. We follow the approach used in \cite{JL04}. 

\subsubsection{Pressure Poisson Formulation}
Taking the divergence of the momentum equation in \eqref{eqn:nse-1}, assuming sufficient regularity, and using $\nabla \cdot u = 0$ gives
\begin{equation}\label{eqn:PPE-strongI}
\Delta p = - \nabla \cdot (u \cdot \nabla u) + \nabla \cdot f.
\end{equation} 
For this equation to be equivalent to the NSE, we need to impose additional constraints on \eqref{eqn:PPE-strongI}. Some possibilities include the enforcement of a no-slip boundary for the divergence of the velocity, retaining the term $\Delta (\nabla \cdot u)$ in \eqref{eqn:PPE-strongI}, or incorporating a Neumann boundary condition into \eqref{eqn:PPE-strongI}. Full details on these different approaches can be found in \cite{GM87,JL04,SSPG06}. 

We will consider the most common approach used in the ROM setting, adding a Neumann boundary condition to \eqref{eqn:PPE-strongI}. To this end, we take the normal component of the momentum equation along the boundary $\Gamma$. Using the vector identity 
\begin{equation*}
\Delta u = -\nabla \times \nabla \times u + \nabla(\nabla \cdot u)       
\end{equation*}
along with $\nabla \cdot u = 0$ gives
\begin{equation}
\frac{\partial p}{\partial n}\biggr|_{\Gamma} = \bigg [-\nu \cdot (\nabla \times \nabla \times u) + n \cdot f \bigg]\biggr|_{\Gamma},
\end{equation}
where $n$ is the unit normal along $\Gamma$. Equipping \eqref{eqn:PPE-strongI} with this boundary condition then gives the full PPE
\begin{subequations}\label{eqn:PPE-strongII}
\begin{align}
&\Delta p = - \nabla \cdot (u \cdot \nabla u) + \nabla \cdot f \label{eqn:PPE-strongIIa}
\\
&\frac{\partial p}{\partial n}\biggr|_{\Gamma} = \bigg [-\nu \cdot (\nabla \times \nabla \times u) + n \cdot f \bigg]\biggr|_{\Gamma}. \label{eqn:PPE-strongIIb}
\end{align}
\end{subequations}
Putting this into a weak formulation, we multiply \eqref{eqn:PPE-strongIIa} by a test function $q$. 
Integrating the left hand side and right hand side of \eqref{eqn:PPE-strongIIa} by parts and applying the vector identity
\begin{equation}
\int_{\Gamma} n\cdot (\nabla \times \nabla \times u) q = - \int_{\Gamma} (\nabla \times u) \cdot (n \times \nabla q)
\end{equation}
gives the weak form of the PPE
\begin{equation}\label{eqn:PPE-weak}
(\nabla p, \nabla q) = - (u \cdot \nabla u,\nabla q) + (f,\nabla q) + \nu \int_{\Gamma} (\nabla \times u) \cdot (n \times \nabla q).
\end{equation}

Equation \eqref{eqn:PPE-weak} can then be discretized using the pressure POD basis along with the discrete velocity solution to recover the pressure at each time step. Specifically, given the ROM velocity solution $u_{r}^{n+1},$ we find $p_{m}^{n+1} \in Q_{m}$ satisfying

\begin{equation}
\begin{aligned}\label{eqn:D-PPE}
(\nabla p^{n+1}_{m},\nabla \psi) &=  -\left( u^{n+1}_{r} \cdot \nabla u_{r}^{n+1},\nabla \psi \right) + (f^{n+1}, \nabla \psi)  
\\
&+ \nu \int_{\Gamma} (\nabla \times u^{n+1}_{r}) \cdot (n \times \nabla \psi) \qquad \qquad \qquad \forall \psi \in Q_{m}.
\end{aligned} 
\end{equation}

\begin{remark}\label{remark:ppe}
For the boundary term appearing in \eqref{eqn:D-PPE} to be well posed, this will require either that $(\nabla \times u^{n+1}_{r}))|_{\Gamma} \in H^{1/2}(\Gamma)$ and $(n \times \nabla \psi) \in H^{-1/2}(\Gamma)$ or that $(\nabla \times u^{n+1}_{r}))|_{\Gamma} \in L^{2}(\Gamma)$ and $(n \times \nabla \psi) \in L^{2}(\Gamma)$.  The first of these conditions will be satisfied if $u_{r}^{n+1} \in H^{2}$ and $\psi \in H^{1}$. Since $u^{n+1}_{r} \in X_{r} \subset X_{h}$ and $\psi \in Q_{m} \subset Q_{h},$ this will not hold when  a $C^{0}$ finite element space is used in the offline phase.  The second condition, however, will be true for $C^{0}$ finite elements. Since $u^{n+1}_{r}$ and $\psi$ will be piecewise polynomials on the boundary, they will be in $L^{2}(\Gamma)$. 

Even though this term will be well defined, it will present difficulties in terms of the  theoretical analysis. In order to obtain stability and error estimates the terms involving the boundary need to be bounded in terms of the domain $\Omega$. A standard finite element approach would be to use a trace inequality (see \cite{BS02}) on these terms. However, due to the lack of regularity of these terms, it is not possible to do so here. To our knowledge, the analysis of this equation, even in the finite element setting, is an open problem.  
\end{remark}
%

\section{Error Analysis}\label{sec:error_analysis}
In this section, we conduct an error analysis for the pressure determined by the MER formulation, \eqref{eqn:posteriori-momentum}. We begin by stating preliminary results and establishing notation.

The following stability result for {BE-ROM}, \eqref{eqn:BE_ROM}, holds.
\begin{lemma}\label{lemma:BE-ROM-stability}
	Consider the method \eqref{eqn:BE_ROM}. Let
	\begin{equation*}
	C_{stab} := \|u^{0}_{r}\|^{2} + \nu^{-1}\Delta t\sum_{n=0}^{N'}\|f^{n+1}\|_{-1}^{2},
	\end{equation*}
	 then for any $1 < N' \leq N $
	\begin{equation}\label{energy_inequality}
	\|u_{r}^{N'}\|^{2} + \nu \Delta t \sum_{n=0}^{N'}\|\nabla u_{r}^{n+1}\|^{2} \leq C_{stab}.
	\end{equation}
\end{lemma}
\begin{proof}
	The results follows by letting $\varphi = u_{r}^{n+1}$ and using Cauchy-Schwarz, skew-symmetry of $b^{\ast}$, Young's inequality, and a polarization identity. 
\end{proof}

\begin{definition}
	{Let $C$ be a constant which may depend on $f,u,p,C_{b^*},\nu, C_{stab}$, but is independent of $h, \Delta t, r,m, , \lambda_{i},\sigma_{i}$.} 
\end{definition} 

The POD mass and stiffness matrices of the velocity space are defined as 
\begin{equation*}
\mathbb{M}_{r} = (\varphi_{i}, \varphi_{j})_{L^{2}}, \ \ \ \mathbb{S}_{r} = (\nabla \varphi_{i}, \nabla \varphi_{j})_{L^{2}}.
\end{equation*}
The following POD inverse estimate then holds: 
\begin{lemma}\label{lemma:POD-inveq}
	For all $\varphi \in X_{r}$ and $\psi \in Q_{m}$ it holds
	\begin{equation*}
	\|\nabla \varphi \| \leq |||\mathbb{S}_{r}|||_{2}^{1/2}\|\varphi\|.
	\end{equation*}
\end{lemma}
\begin{proof}
	See Lemma 2 of \cite{KV01}.
\end{proof}


We next define the $L^{2}$ projection into the velocity space $X_{r}$, and the pressure space $Q_{m}$.
\begin{definition} We define the $L^{2}$ projection into the velocity space $X_{r}$, and the pressure space $Q_{m}$ as $P_{r}: L^{2}(\Omega) \rightarrow X_{r}$ and $\chi_{m}: L^{2}(\Omega) \rightarrow Q_{m}$ such that
	\begin{equation}
	\begin{aligned}
	(u - P_{r}u,\varphi) &= 0, \qquad \forall \varphi \in X_{r}, \ \ \text{and} \\
	(p - \chi_{m}p, \psi) &= 0, \qquad \forall \psi \in Q_{m}.
	\end{aligned}
	\end{equation}
\end{definition}

%
%

The following lemmas, proven in \cite{KV01,S14}, 
provide bounds for the error between the snapshots and their projections onto the POD space. 

\begin{lemma}\label{v-proj-errL2} It holds that
	\begin{equation}
	\begin{aligned}
	&\frac{1}{N+1}\sum_{n = 0}^{N} \left \|u_{h}^{n} - \sum_{i=1}^{r}(u_{h}^{n},{\varphi_i}){\varphi}_i \right \|^{2}  = \sum_{i=r+1}^{N_{V}} {\lambda_i},\ \ \text{and} \\
	&\frac{1}{N+1}\sum_{n = 0}^{N} \left \| p_{h}^{n} - \sum_{i=1}^{m}(p_{h}^{n},\psi_i)\psi_i \right \|^{2} = \sum_{i = m + 1}^{N_P} \sigma_i .
	\end{aligned}
	\end{equation}
\end{lemma}


We also have the following $H^{1}$ error bound for the velocity.

\begin{lemma}\label{v-proj-errH1}It holds that 
	\begin{equation}
	\begin{aligned}
	&\frac{1}{N+1}\sum_{n = 0}^{N} \left \|\nabla(u_{h}^{n} - \sum_{i=1}^{r}(u_{h}^{n},{\varphi_i}){\varphi}_i) \right \|^{2}  =  \sum_{i=r+1}^{N_{V}} \|  \nabla {\varphi}_i\| ^2\lambda_i.
	\end{aligned}
	\end{equation}
\end{lemma}

From these projection estimates we can derive error estimates for the $L^{2}$ projection error into the velocity space $X_{r}$ using the approach of Lemma 3.3 in \cite{IW14}.
\begin{lemma}\label{pod-velo-proj-lemma-L2}
	For any $u^{n} \in V$ the $L^{2}$ projection error into $X_{r}$  satisfies the following estimates
	\begin{equation}\label{proj-err-1}
	\begin{aligned}
	&\frac{1}{N+1}\sum_{n=0}^{N} \|u^{n} - P_{r}u^{n}\|^{2} \leq C(\nu,p) \left (h^{2s} + \Delta t^{2} + \sum_{i= r+1}^{N_V} {\lambda}_{i} \right), \ \ \text{and} \\
	& \frac{1}{N+1}\sum_{n=0}^{N} \|\nabla(u^{n} - P_{r}u^{n})\|^{2} \leq {C(\nu,p)}\bigg((1+|||{\mathbb S}_r|||_{2}) h^{2s}   + (1+|||{\mathbb S}_r|||_{2})\Delta t^2
	\\ & \hspace{4cm} + \sum_{i=r+1}^{N_{V}} \|  \nabla {\varphi}_i\| ^2\lambda_i\bigg) . 
	\end{aligned}
	\end{equation}
\end{lemma}

A similar results holds for the for the $L^{2}$ projection error into the pressure space $Q_{m}$. 
\begin{lemma}\label{pod-press-proj-lemma-L2-basis}
	For any $p^{n} \in Q$ the $L^{2}$ projection error satisfies the following estimates
	\begin{equation}\label{proj-err-2}
	\begin{aligned}
	&\frac{1}{N+1}\sum_{n=0}^{N} \|p^{n} - \chi_{m}p^{n}\|^{2} \leq C(\nu,p) \left (h^{2k} + \Delta t^{2} + \sum_{i= m+1}^{N_{P}} \sigma_{i} \right).
	\end{aligned}
	\end{equation}
\end{lemma}


To prove pointwise in time error estimates for the velocity, we must make the following assumption similar to the one stated in \cite{IW14}.

\begin{assumption}\label{assumption:conv}
		For any $u^{n} \in V,$ the $L^{2}$ projection error into $X_{R}$  satisfies the following estimates
		\begin{equation}\label{proj-err-v-assump}
		\begin{aligned}
		&	\max_{n} \|u^{n} - P_{r}u^{n}\|^{2} \leq C(\nu,p) \left (h^{2s} + \Delta t^{2} + \sum_{i= r+1}^{N_V} {\lambda}_{i} \right), \ \ \text{and} \\
		&\max_{n} \|\nabla(u^{n} - P_{r}u^{n})\|^{2} \leq {C(\nu,p)}\bigg((1+|||{\mathbb S}_r|||_{2}) h^{2s}   + (1+|||{\mathbb S}_r|||_{2})\Delta t^2
		\\ & \hspace{4cm} + \sum_{i=r+1}^{N_{V}} \|  \nabla {\varphi}_i\| ^2\lambda_i\bigg) .  \\
		\end{aligned}
		\end{equation}
\end{assumption}

 We denote by $e_{u}$ and $e_{p}$ the error between the true velocity and pressure solution and their respective POD approximations.  We then split the error for the velocity and pressure via the $L^{2}$ projection into the space $X_{r}$ and $Q_{m},$ respectively
\begin{equation*}
\begin{aligned}
e^{n+1}_{u} = u^{n+1} - u^{n+1}_{r} = (u^{n+1} - P_{r}(u^{n+1})) + (P_{r}(u^{n+1}) - u^{n+1}_{R}) &= \eta^{n+1} - \xi_{r}^{n+1}
\\ e^{n+1}_{p} = p^{n+1} - p^{n+1}_{m} = (p^{n+1} - \chi_{m}(p^{n+1})) + (\chi_m(p^{n+1}) - p^{n+1}_{m}) &= \kappa^{n+1} - \pi_{m}^{n+1}.
\end{aligned}
\end{equation*}

Lastly, we state  a convergence result for the velocity determined by the BE-ROM scheme \eqref{eqn:BE_ROM}.

\begin{theorem}\label{theorem:velocity_err}
	Consider BE-ROM \eqref{eqn:BE_ROM} and let $C$ be a constant which may depend on $f,u,p,C_{b^{\ast}},C_{stab}$ and, $\nu$, but is independent of $h, \Delta t,r,m, \lambda_{i},$ and ${\mathbb S}_r$.   Under the regularity conditions from Assumption \ref{assumption:regularity} and the projection error estimates from Assumption \ref{assumption:conv}, for any $0 \leq n \leq N$, the following bound on the velocity error holds
	\begin{equation}
	\|e_{u}^{n+1}\|^{2} + \nu \vvvert \nabla e_{u} \vvvert_{2,0}^{2} \leq C\left((1 + |||{\mathbb S}_r|||_{2}) (h^{2s} + \dt^{2})  + \sum_{i=r+1}^{N_{V}}\lambda_i + \sum_{i=r+1}^{N_{V}}\lambda_i \|\nabla \varphi_{i}\|^{2}\right) .
	\end{equation}
	\begin{proof}
	The proof is identical to that of Theorem 4.1 in \cite{MRXI17}.
	\end{proof}
\end{theorem}

\subsection{Momentum Equation Stability and Error Analysis}

Next, we conduct a full stability and error analysis for the MER formulation \eqref{eqn:posteriori-momentum}. We begin by stating some preliminary definitions and lemmas. 

The spaces $X_{r}$ and $S_{m}$ have the following dual norms 
\begin{equation*}
\|w\|_{X^{\ast}_{r}}:= \sup_{\varphi \in X_{r}} \frac{(w,\varphi)}{\| \nabla \varphi \|}  \qquad \|w\|_{S^{\ast}_{m}}:= \sup_{\zeta \in S_{m}} \frac{(w,\zeta)}{\| \nabla \zeta\|}.
\end{equation*}

We recall the strengthened Cauchy-Buniakowskii-Schwarz (CBS). This inequality has been used in the analysis for multilevel schemes \cite{EV91}  and recently in the analysis of ROMs \cite{DILMS19,M14,R19}.

\begin{lemma}\label{CBS}
	Given a Hilbert space V and two finite dimensional subspaces  $V_{1} \subset V$ and $V_{2} \subset V$ with trivial intersection: 
	\begin{equation*}
	V_{1} \cap V_{2} = \{0\},
	\end{equation*}
	then there exists $ 0 \leq \alpha < 1$ such that
	\begin{equation*}
	|(v_{1},v_{2})| \leq \alpha \|v_{1}\|\|v_{2}\| \ \ \ \forall v_1 \in V_{1}, v_2 \in V_{2}.
	\end{equation*} 
\end{lemma}

In the ensuing analysis we will be interested in computing the value of $\alpha$ between the spaces $X_{r}$ and $S_{m}$. This can also be interpreted as determining the first principal angle defined as

	\begin{equation}\label{first_principal_angle}
\theta_{1} := \min_{\varphi \neq 0, \zeta \neq 0} \left\{\arccos\left({\frac{|(\varphi,\zeta)|}{\|\varphi\|\|\zeta\|}}\right) \bigg| \varphi \in X_{r}, \zeta \in S_{m} \right\},
\end{equation} 
with $0 < \theta_{1} \leq \frac{\pi}{2}$. 

Numerous methods for calculating the principal angle between two spaces using either a QR or SVD factorization have been devised in \cite{KA02,WS03} and the references therein. We note that due to the relative small size of the reduced basis, this computation is negligible in terms of computational cost and storage.

Next, we prove an $H^{1}$ stability results for the $L^{2}$ projection from $X_{r}$ into $S_{m}$. In the finite element setting, this type of result is known to hold independent of the cardinality of the basis for quasi-uniform and certain regular meshes \cite{BY14}. In the ROM setting, however, this is currently an open problem (see Remark 4.1 in \cite{XWWI18}). 
\begin{lemma}
	Let $u_{m} \in S_{m}$ and $P_{r}:S_{m} \rightarrow X_{r}$ denote the $L^{2}$ projection from $S_{m}$ to $X_{r}$. Letting 
	\begin{equation*}
	C_{r}^{H^{1}}  := \left\|\sum_{i=1}^{r} \nabla \varphi_{i} \right\|,
	\end{equation*}
	 the following stability bound holds
	\begin{equation}
	\|\nabla P_{r} u_{m}\| \leq \alpha C_{P} C_{r}^{H^{1}} \|\nabla u_{m}\|.
	\end{equation}

\end{lemma}
	\begin{proof}
		By the definition of the $L^{2}$ projection into $X_{r}$ we have 
		\begin{equation}
		\|\nabla P_{r} u_{m}\| = \left\|\sum_{i=1}^{r}(u_{m},\varphi_{i}) \nabla \varphi_{i} \right\|.
		\end{equation}
		Since $X_{r} \subset (V^{div}_{h})$ and $S_{m} \subset (V^{div}_{h})^{\perp}$ it follows that $X_{r} \cap S_{m} = \{0\}$. Therefore, by Lemma \ref{CBS} it follows that
		\begin{equation}
		\left\|\sum_{i=1}^{r}(u_{m},\varphi_{i}) \nabla \varphi_{i} \right\| \leq \alpha \|u_{m}\| \left\|\sum_{i=1}^{r}\|\varphi_{i}\| \nabla \varphi_{i} \right\| .
		\end{equation}
		Then by the $L^{2}$ orthonormality of the basis and Poincar\'{e} inequality we have
		\begin{equation}
		\alpha \|u_{m}\| \left\|\sum_{i=1}^{r}\|\varphi_{i}\| \nabla \varphi_{i} \right\| \leq \alpha C_{P} \left\|\sum_{i=1}^{r} \nabla \varphi_{i} \right\|\|\nabla u_{m}\| .
		\end{equation}
		\end{proof} 

Unlike the finite element setting, this stability result indicates that the bound will not be independent of the number of POD basis functions used. However, if $\alpha$ is sufficiently small; i.e., $\theta_{1}$ is close to $\pi/2$ indicating that the spaces $X_{r}$ and $S_{m}$ are nearly orthogonal in the $L^{2}$ sense, then the stability bound will be well behaved.  

Using this stability result we prove a bound on the dual norm of $S^{\ast}_{m}$ in terms of $X^{\ast}_{r}$.

\begin{lemma}\label{lemma:eqvn_norms}
	Let $u_{r} \in X_{r}$, the following bound will then hold between the dual norms 
	\begin{equation}
     \|u_{r}\|_{S^{\ast}_{m}}  \leq \alpha C_{P} C_{r}^{H^{1}} \|u_{r}\|_{X^{\ast}_{r}}.  
	\end{equation}
\end{lemma}
\begin{proof}
\begin{equation}
\begin{aligned}
	\|u_{r}\|_{S^{\ast}_{m}} &= \sup_{\zeta \in S_{m}} \frac{(u_{r},\zeta)}{\| \nabla \zeta \|} 
	\\
	&=  \sup_{\zeta \in S_{m}} \frac{(u_{r}, P_{r}\zeta + P^{\perp}_{r}\zeta)}{\| \nabla \zeta \|}  
	\\
	&= \sup_{\zeta \in S_{m}} \frac{(u_{r}, P_{r}\zeta )}{\| \nabla \zeta \|} 
	\\
	&\leq {\alpha C_{P} C_{r}^{H^{1}}} \sup_{\zeta \in S_{m}} \frac{(u_{r}, P_{r}\zeta )}{\|P_{r} \nabla \zeta \|} 
	\\
	& \leq {\alpha C_{P} C_{r}^{H^{1}}} \sup_{\varphi \in X_{r}} \frac{(u_{r}, \varphi )}{\|\nabla \varphi \|} =  {\alpha C_{P} C_{r}^{H^{1}}} { \|u_{r}\|_{X^{\ast}_{r}}}.
	\end{aligned}
\end{equation}	
\end{proof}

Next, we give an $L^{1}(0,T,L^{2}(\Omega))$ stability result for the pressure determined via the MER formulation.
\begin{theorem}\label{theorem:stability_ME}
	Consider the pressure approximation determined from \eqref{eqn:posteriori-momentum}. \newline The following energy inequality holds
	\begin{equation}\label{eqn:stability_mom_L2}
	\begin{aligned}
	\beta_{m}|||p_{m}|||_{1,0}&\leq  \left(1 + {\alpha C_{P} C_{r}^{H^{1}}}\right)\bigg(\cb C_{stab}\nu^{-1}  + \dt \sum_{n=0}^{N}\|f^{n+1}\|_{-1}\bigg)  
	\\
	&+{\alpha C_{P} C_{r}^{H^{1}}} \sqrt{\nu T C_{stab}}.
	\end{aligned}
	\end{equation}
\end{theorem}
\begin{proof}
	We follow a similar proof path to that in \cite{J18}. Let $\varphi \in X_{r}$, then taking equation \eqref{eqn:BE_ROM} and isolating the time derivative gives
	\begin{equation}\label{theorem:stability_ME-eqn1}
	\Big(\frac{u^{n+1}_r - u_{r}^{n}}{\Delta t}, \varphi \Big) = ( f^{n+1}, \varphi) -  b^{\ast}(u_r^{n} , u^{n+1}_r ,\varphi) - \nu (\nabla u^{n+1}_r, \nabla \varphi).  
	\end{equation}
	Standard bounds on the right hand side yield 
	\begin{equation}\label{theorem:stability_ME-eqn2}
	\begin{aligned}
	-  b^{\ast}(u_r^{n} , u^{n+1}_r ,\varphi) &\leq \cb \|\nabla u_r^{n} \|\|\nabla u_r^{n+1} \|\|\nabla \varphi \| \\
	- \nu (\nabla u^{n+1}_r, \nabla \varphi) &\leq \nu \|\nabla u^{n+1}_{r}\|\|\nabla \varphi \| \\
	( f^{n+1}, \varphi) &\leq \| f^{n+1}\|_{-1} \| \nabla \varphi \|.
	\end{aligned}
	\end{equation}
	It then follows, using these estimates, dividing both sides by $\|\nabla \varphi\|$ and taking the supremum over $\varphi \in X_{r}$ that
	\begin{equation}\label{theorem:stability_ME-eqn3}
	\left\|\frac{u^{n+1}_r - u_{r}^{n}}{\Delta t} \right\|_{X^{\ast}_{r}} \leq \cb \|\nabla u^{n}_{r} \| \|\nabla u^{n+1}_{r} \| + \nu\|\nabla u^{n+1}_{r} \| + \|f^{n+1}\|_{-1}.
	\end{equation}
	Using Lemma \ref{lemma:eqvn_norms} we then
	have
	\begin{equation}\label{theorem:stability_ME-eqn4}
	\left\|\frac{u^{n+1}_r - u_{r}^{n}}{\Delta t} \right\|_{S^{\ast}_{m}} \leq {\alpha C_{P} C_{r}^{H^{1}}} \left( (\cb\|\nabla u^{n}_{r} \| + \nu)\|\nabla u^{n+1}_{r} \| + \|f^{n+1}\|_{-1}\right).
	\end{equation}
	Now considering \eqref{eqn:posteriori-momentum} and using the bounds from \eqref{theorem:stability_ME-eqn2}
	\begin{equation}\label{theorem:stability_ME-eqn5}
	\begin{aligned}
	(p_{m}^{n+1}, \nabla \cdot \zeta) &\leq \Big(\frac{u^{n+1}_r - u_{r}^{n}}{\Delta t}, \zeta \Big) 
	+ \cb \|\nabla u^{n}_{r} \| \|\nabla u^{n+1}_{r} \|\|\nabla \zeta \| + \|\nabla \zeta \|\|f^{n+1}\|_{-1}.
	\end{aligned}
	\end{equation}
	Dividing both sides by $\|\nabla \zeta \|$, taking the supremum over $\zeta \in S_{m}$, and using the discrete inf-sup condition from Lemma \ref{lemma:inf-sup} and estimate \eqref{theorem:stability_ME-eqn4} gives
	\begin{equation}
	\begin{aligned}
	\beta_{m}\|p_{m}^{n+1}\| &\leq \left(1 + \alpha C_{P} C_{r}^{H^{1}}\right) \left(\cb \|\nabla u^{n}_{r} \| \|\nabla u^{n+1}_{r} \|+  \|f^{n+1}\|_{-1}\right) 
	\\
	&+ \alpha C_{P} C_{r}^{H^{1}}\nu\|\nabla u^{n+1}_{r} \|. 
	\end{aligned}
	\end{equation}
	Multiplying by $\dt$ and summing from $n=0$ to $n=N$ then yields
	\begin{equation}
	\begin{aligned}
	\beta_{m} \dt \sum_{n=0}^{N}\|p_{m}^{n+1} \| \leq &\left(1 + \alpha C_{P} C_{r}^{H^{1}}\right) \times \biggr(\cb \dt \sum_{n=0}^{N}\|\nabla u^{n}_{r} \| \|\nabla u^{n+1}_{r} \|  
	\\
	&+  \dt \sum_{n=0}^{N}\|f^{n+1}\|_{-1}\biggr)+ \alpha C_{P} C_{r}^{H^{1}}\nu \dt \sum_{n=0}^{N}\|\nabla u_{r}^{n+1}\|.
	\end{aligned}
	\end{equation}
	Bounding the terms on the right-hand side by Cauchy-Schwarz, Young's inequality, and Lemma \ref{lemma:BE-ROM-stability}
	\begin{equation}
	\begin{aligned}
	&\cb \dt \sum_{n=0}^{N}\|\nabla u^{n}_{r} \| \|\nabla u^{n+1}_{r} \|   \leq \frac{\cb \dt}{2} \sum_{n=0}^{N}\|\nabla u^{n+1}_{r} \|^{2} + \frac{\cb \dt}{2} \sum_{n=0}^{N}\|\nabla u^{n}_{r} \|^{2} \leq \frac{\cb C_{stab}}{\nu} 
	\\
	&\nu \dt \sum_{n=0}^{N}\|\nabla u_{r}^{n+1}\| \leq \sqrt{\nu T} \sqrt{\nu \dt \sum_{n=0}^{N}\|\nabla u_{r}^{n+1}\|^{2}} \leq \sqrt{\nu T C_{stab}}.
	\end{aligned}
	\end{equation}
	Combining and simplifying terms \eqref{eqn:stability_mom_L2} follows. 
\end{proof}

According to Theorem \ref{theorem:stability_ME}, if the product $\alpha C_{r}^{H^{1}}$ is sufficiently small, the stability estimate for the pressure will scale similarly to the velocity determined by the BE-ROM scheme.

Finally, we state the main result of this section, an $L^{1}(0,T,L^{2}(\Omega))$ convergence  result for the pressure determined via the MER formulation.
\begin{theorem}\label{theorem:pressure_convergence}
	Consider the MER scheme \eqref{eqn:posteriori-momentum} and {BE-ROM} \eqref{eqn:BE_ROM}. Under the regularity conditions made in Assumption \ref{assumption:regularity}, the following bound on the pressure error holds
\begin{equation} 
\begin{aligned}
\beta_m &|||e_{p}|||_{1,0}
\\ 
&\leq  C\biggr[ (1+\beta_m)\sqrt{T} \vvvert \kappa \vvvert_{2,0} 
+ \dt \|\eta_{t} \|_{L^2(0,T,L^2(\Omega))} +
\dt^{3/2} \|\eta_{tt} \|_{L^2(0,T,L^2(\Omega))}
\\
&
+ \alb \biggr( \dt^{3/2} \|  u_{tt} \|_{L^2(0,T,L^2(\Omega))} +
\dt^{2} \|\nabla u_{t} \|_{L^2(0,T,L^2(\Omega))}+
\\
& + \dt^{5/2} \|\nabla u_{tt} \|_{L^2(0,T,L^2(\Omega))}+ \Big(\sqrt{T} + {C_{stab}} \Big)\vvvert\nabla e_{u}   \vvvert_{2,0}
\biggr) \biggr] .
\end{aligned}
\end{equation}
	\begin{proof}
		The weak solution of the NSE satisfies
		\begin{equation}
		\begin{aligned}
		\label{er-eq1-lc}
		\left(u_t^{n+1}, \varphi \right) + \bstar{u^{n+1}}{u^{n+1}}{\varphi} + \nu(\nabla u^{n+1},\Grad{\varphi}) = (f^{n+1},\varphi). 
		\end{aligned}
		\end{equation}	
		Subtracting \eqref{eqn:BE_ROM} from \eqref{er-eq1-lc} yields
		\begin{equation}
		\begin{aligned}
		\Big(\frac{e_u^{n+1} - e_u^{n}}{\Delta t}, \varphi \Big) + 
		b^{\ast}(&u^{n+1}- u^{n} , u^{n+1} ,\varphi) 
		+b^{\ast}(e_u^{n}, u^{n+1} ,\varphi)
		\\ &+b^{\ast}(u_r^{n} , e_u^{n+1} ,\varphi)
		+\nu (\nabla e_u^{n+1}, \nabla \varphi) 
		=\Big(\frac{u^{n+1} - u^{n}}{\Delta t}-u_t^{n+1}, \varphi \Big).
		\end{aligned}
		\end{equation}
		Splitting the error, using the fact that $\left( \frac{\eta^{n+1} - \eta^{n}}{\Delta t} ,\varphi \right) = 0$ by the definition of the $L^{2}$ projection,  and  rearranging terms gives
		\begin{equation}
		\begin{aligned}
		\Big(\frac{\xi^{n+1} - \xi^{n}}{\Delta t}, \varphi \Big)& =  
		\nu (\nabla e_u^{n+1}, \nabla \varphi)
		-\Big(\frac{u^{n+1} - u^{n}}{\Delta t}-u_t^{n+1}, \varphi \Big)
		\\ 
		&+  b^{\ast}(u^{n+1}- u^{n} , u^{n+1} ,\varphi) 
		+ b^{\ast}(e_u^{n}, u^{n+1} ,\varphi) +
		b^{\ast}(u_r^{n} , e_u^{n+1} ,\varphi) . 
		\end{aligned}
		\end{equation}
		Applying Cauchy-Schwarz, Taylor's Theorem, Poincar\'{e} inequality, and Lemma \ref{lemma:trilinear} to the terms on the right hand side yields
		\begin{equation}\label{a_1}
		\begin{aligned}
		\nu (\nabla e_u^{n+1}, \nabla \varphi) &\leq \nu \| \nabla e_u^{n+1} \| \| \nabla \varphi \|
		\\
		\Big(\frac{u^{n+1} - u^{n}}{\Delta t}-u_t^{n+1}, \varphi \Big) &\leq C C_{P} \sqrt{\Delta t} \|  u_{tt} \|_{L^2(t^n,t^{n+1},L^2(\Omega))} \| \nabla \varphi \|
		\\
		b^{\ast}(e_u^{n}, u^{n+1} ,\varphi) &\leq C_{b^{\ast}} \| \nabla e_u^n \|\| \nabla u^{n+1} \| \| \nabla \varphi \|
		\\
		b^{\ast}(u_r^{n} , e_u^{n+1} ,\varphi) &\leq C_{b^{\ast}} \| \nabla u_r^n \|\| \nabla e_u^{n+1} \| \| \nabla \varphi \|
		\\
		b^{\ast}(u^{n+1}- u^{n} , u^{n+1} ,\varphi) &\leq C\Delta t^{3/2} \| \nabla u_{tt} \|_{L^2(t^n,t^{n+1},L^2(\Omega))}\| \nabla u^{n+1} \| \| \nabla \varphi \|
		\\
		&+ C \Delta t  \| \nabla u_{t} \|_{L^2(t^n,t^{n+1},L^2(\Omega))}\| \nabla u^{n+1} \| \| \nabla \varphi \|.
		\end{aligned}
		\end{equation}
		Next, dividing by $\| \nabla \varphi \|$ and taking the supremum over all $ \varphi \in X_r$, gives a bound on the dual norm $X^{\ast}_{r}$
		\begin{equation}\label{eqn:xi_bound1}
		\begin{aligned}
		\Big\|\frac{\xi_{r}^{n+1} - \xi_{r}^{n}}{\Delta t} \Big\|_{X^{\ast}_r} \leq
		& \nu\| \nabla e_u^{n+1}\|+ C C_{P} \sqrt{\Delta t} \| u_{tt} \|_{L^2(t^n,t^{n+1},L^2(\Omega))}+
		C_{b^{\ast}}\| \nabla e_u^{n+1} \|\| \nabla u_r^{n}\| +
		\\&
		\| \nabla u^{n+1} \| \Big(\Delta t \| \nabla u_t \|_{L^2(t^n,t^{n+1},L^2(\Omega))}+
		\\
		&\Delta t^{3/2} \| \nabla u_{tt} \|_{L^2(t^n,t^{n+1},L^2(\Omega))}
		+ \cb  \| \nabla e_u^n \| \Big) . 
		\end{aligned}
		\end{equation}
		Using Lemma \ref{lemma:eqvn_norms} then yields a bound on the dual norm $S_{m}^{\ast}$
		\begin{equation}\label{eqn:xi_bound2}
		\Big\|\frac{\xi_{r}^{n+1} - \xi_{r}^{n}}{\Delta t} \Big\|_{S^{\ast}_m} \leq
		{\alpha C_{P} C_{r}^{H^{1}}} \Big\|\frac{\xi_{r}^{n+1} - \xi_{r}^{n}}{\Delta t} \Big\|_{X^{\ast}_r}. 
		\end{equation}
		Next, we consider the weak form of the NSE, with a test function $\zeta \in S_m$
		\begin{equation}\label{cts press}
		\begin{aligned}
		(p^{n+1}, \nabla \cdot \zeta) =  (u_t, \zeta ) + b^{\ast}(u^{n+1} , u^{n+1},\zeta) 
		&+\nu (\nabla u^{n+1}, \nabla \zeta)  - ( f^{n+1}, \zeta)  \ \  \forall \zeta \in S_m.
		\end{aligned}
		\end{equation} 
		Subtracting \eqref{eqn:posteriori-momentum} from \eqref{cts press} splitting the pressure error, and adding and subtracting $\eta^{n+1}_{t}$ gives
		%
		\begin{equation}
		\begin{aligned}
		(\pi_{m}^{n+1},  \nabla \cdot \zeta) &=   (\kappa^{n+1}, \nabla \cdot \zeta) -  \nu (\nabla e_u^{n+1}, \nabla \zeta) - 
		b^{\ast}(u^{n+1}-u^n ,u^{n+1},\zeta) 
		\\
		& -
		b^{\ast}(e_u^{n} , u_r^{n+1},\zeta)   -b^{\ast}(u_r^{n} , e_u^{n+1},\zeta) 		
		- \Big(u_{t} - \frac{u^{n+1} - u^{n}}{\Delta t}, \zeta \Big) 
		\\
		&- \Big(\frac{ \eta^{n+1} - \eta^{n}}{\Delta t} - \eta^{n+1}_{t}, \zeta \Big) -\Big(\eta^{n+1}_{t}, \zeta \Big)  + \Big(\frac{\xi_{r}^{n+1} - \xi_{r}^{n}}{\Delta t}, \zeta \Big).  
		\end{aligned}
		\end{equation} 
		The first eight terms on the right hand side are bounded using Cauchy-Schwarz, Taylor's Theorem, the Poincar\'{e} inequality, and Lemma \ref{lemma:trilinear}
		\begin{equation}
		\begin{aligned}
		(\kappa^{n+1}, \nabla \cdot \zeta) \leq& \sqrt{d} \| \kappa^{n+1}\| \| \nabla \zeta \| 
		\\
		-  \nu (\nabla e_u^{n+1}, \nabla \zeta) \leq& \nu \|\nabla e_u^{n+1}\|\|\nabla \zeta\|
		\\
		-b^{\ast}(e_u^{n}, u^{n+1} ,\zeta) \leq& C_{b^{\ast}} \| \nabla e_u^n \|\| \nabla u^{n+1} \| \| \nabla \zeta \|
		\\
		-b^{\ast}(u_r^{n} , e_u^{n+1} ,\zeta) \leq& C_{b^{\ast}} \| \nabla u_r^n \|\| \nabla e_u^{n+1} \| \| \nabla \zeta \|
		\\
		-b^{\ast}(u^{n+1}- u^{n} , u^{n+1} ,\zeta) \leq& C\Delta t^{3/2} \| \nabla u_{tt} \|_{L^2(t^n,t^{n+1},L^2(\Omega))}\| \nabla u^{n+1} \| \| \nabla \zeta \|
		+ \\
		&C \Delta t  \| \nabla u_{t} \|_{L^2(t^n,t^{n+1},L^2(\Omega))}\| \nabla u^{n+1} \| \| \nabla \zeta \|
		\\
		-\Big(\frac{u^{n+1} - u^{n}}{\Delta t}-u_t^{n+1}, \zeta \Big) \leq& C C_{P} \sqrt{\Delta t} \|  u_{tt} \|_{L^2(t^n,t^{n+1},L^2(\Omega))} \| \nabla \zeta\|
		\\
		-\Big(\frac{\eta^{n+1} - \eta^{n}}{\Delta t}-\eta_t^{n+1}, \zeta \Big) \leq& C C_{P} \sqrt{\Delta t} \|  \eta_{tt} \|_{L^2(t^n,t^{n+1},L^2(\Omega))} \| \nabla \zeta\|
		\\
		-\Big(\eta^{n+1}_{t}, \zeta \Big) \leq& C_{P} \|  \eta_{t} \|_{L^2(t^n,t^{n+1},L^2(\Omega))} \| \nabla \zeta\|.
		\end{aligned}
		\end{equation}
		Now, applying these bounds, dividing by $\| \nabla \zeta \|$, and taking the supremum over all $\zeta \in S_m$ gives
		\begin{equation}
		\begin{aligned}
		\sup_{\zeta \in S_m}& \frac{(\pi_m^{n+1}, \nabla \cdot \zeta)}{\| \nabla \zeta \|} \leq 
		\sqrt{d} \| \kappa^{n+1}\|  +  \nu \| \nabla e_u^{n+1} \| 
		+ C_{b^{\ast}} \| \nabla e_u^{n}\| \| \nabla u^{n+1}\|
		\\
		&+ C_{b^{\ast}} \| \nabla u^{n}_r\| \| \nabla e_u^{n+1}\| 
		+ 
		C\Delta t^{3/2} \| \nabla u_{tt} \|_{L^2(t^n,t^{n+1},L^2(\Omega))}\| \nabla u^{n+1} \|
		+ \\
		&C \Delta t  \| \nabla u_{t} \|_{L^2(t^n,t^{n+1},L^2(\Omega))}\| \nabla u^{n+1} \| + 
		CC_{P} \sqrt{\Delta t} \|  u_{tt} \|_{L^2(t^n,t^{n+1},L^2(\Omega))} +
		\\ &
		CC_{P} \sqrt{\Delta t} \|  \eta_{tt} \|_{L^2(t^n,t^{n+1},L^2(\Omega))}  
		+  C_P \| \eta_t \|_{L^2(t^n,t^{n+1},L^2(\Omega))} 
		+ \Big\| \frac{\xi^{n+1} - \xi^{n}}{\Delta t} \Big\|_{S^{\ast}_m}.
		\end {aligned}
		\end{equation}
		Recalling from Lemma \ref{lemma:inf-sup} that $S_m$ and $Q_m$ are inf-sup stable with constant $\beta_m$ and using the bound on $ \Big\| \frac{\xi^{n+1} - \xi^{n}}{\Delta t} \Big\|_{S^{\ast}_m}$ from \eqref{eqn:xi_bound1}-\eqref{eqn:xi_bound2}  yields
		\begin{equation}\label{error1}
		\begin{aligned}
		\beta_m \| \pi_m^{n+1} \| 
		\leq & \sqrt{d} \| \kappa^{n+1}\|  
		+ (1+ {\alpha C_{P} C_{r}^{H^{1}}}) \Big[ 
		\nu\| \nabla e_u^{n+1}\|
		+ C C_{P} \sqrt{\Delta t} \| u_{tt}\|_{L^2(t^n,t^{n+1},L^2(\Omega))}
		\\
		& + C_{b^{\ast}}\| \nabla e_u^n \|\| \nabla u_r^{n+1}\| 
		+ \| \nabla u^{n+1} \| \Big(\Delta t \| \nabla u_t \|_{L^2(t^n,t^{n+1},L^2(\Omega))}
		\\
		&+ \Delta t^{3/2} \| \nabla u_{tt} \|_{L^2(t^n,t^{n+1},L^2(\Omega))}
		+ \cb  \| \nabla e_u^{n} \| \Big)\Big] + 
		\\ &
		C_P \| \eta_t \|_{L^2(t^n,t^{n+1},L^2(\Omega))} 
		+  CC_{P} \sqrt{\Delta t} \|  \eta_{tt} \|_{L^2(t^n,t^{n+1},L^2(\Omega))}.
		\end{aligned}
		\end{equation}
		Now, multiplying by $\Delta t$, taking a maximum $C$ over all constants, using the regularity from Assumption \ref{assumption:regularity}, summing from $n = 0$ to $n =N-1$, using Cauchy-Schwarz, and the fact that $\vvvert \nabla u_r \vvvert_{2,0} \leq \sqrt{\frac{ C_{stab}}{\nu}}$ by Lemma \ref{lemma:BE-ROM-stability} we have
		\begin{equation} 
		\begin{aligned}
		\beta_m &\Delta t  \sum_{n = 0}^{N-1} \| \pi_m^{n+1} \| 
		\\ 
		&\leq  C\biggr[ \sqrt{T} \vvvert \kappa \vvvert_{2,0} 
		+ \dt \|\eta_{t} \|_{L^2(0,T,L^2(\Omega))} +
		\dt^{3/2} \|\eta_{tt} \|_{L^2(0,T,L^2(\Omega))}+
		\\
		&
		\alb \biggr( \dt^{3/2} \|  u_{tt} \|_{L^2(0,T,L^2(\Omega))} +
		\dt^{2} \|\nabla u_{t} \|_{L^2(0,T,L^2(\Omega))}
		\\
		& + \dt^{5/2} \|\nabla u_{tt} \|_{L^2(0,T,L^2(\Omega))}+\Big(\sqrt{T} +\sqrt{C_{stab}}\Big)\vvvert\nabla e_{u}   \vvvert_{2,0}
		\biggr) \biggr] .
		\end{aligned}
		\end{equation}
		By the triangle inequality we  have
		\begin{equation}
		\beta_m \Delta t  \sum_{n = 0}^{N-1} \| e_p^{n+1} \| \leq \beta_m \Delta t  \sum_{n = 0}^{N-1} \| \pi_m^{n+1} \| + \beta_m \Delta t  \sum_{n = 0}^{N-1} \| \kappa^{n+1} \|.
		\end{equation}
		Then, applying Cauchy-Schwarz on the second term
		\begin{equation}
		\beta_m \Delta t  \sum_{n = 0}^{N-1} \| \kappa^{n+1} \| \leq \beta_m \Delta t \sqrt{N} \sqrt{\sum_{n = 0}^{N-1} \| \kappa^{n+1} \|^{2}} = \beta_m \sqrt{T}\vvvert \kappa\vvvert_{2,0}.
		\end{equation}
		This then yields the estimate
		\begin{equation} 
		\begin{aligned}
		\beta_m &\Delta t  \sum_{n = 0}^{N-1} \| e_p^{n+1} \| 
		\\ 
		&\leq  C\biggr[ (1+\beta_m)\sqrt{T} \vvvert \kappa \vvvert_{2,0} 
		+ \dt \|\eta_{t} \|_{L^2(0,T,L^2(\Omega))} +
		\dt^{3/2} \|\eta_{tt} \|_{L^2(0,T,L^2(\Omega))}+
		\\
		&
		\alb \biggr( \dt^{3/2} \|  u_{tt} \|_{L^2(0,T,L^2(\Omega))} +
		\dt^{2} \|\nabla u_{t} \|_{L^2(0,T,L^2(\Omega))}+
		\\
		& + \dt^{5/2} \|\nabla u_{tt} \|_{L^2(0,T,L^2(\Omega))}+ \Big(\sqrt{T} + \sqrt{C_{stab}} \Big)\vvvert\nabla e_{u}   \vvvert_{2,0}
		\biggr) \biggr] .
		\end{aligned}
		\end{equation}
		\end{proof}
	\end{theorem}

\begin{corollary}\label{corollary:convergence}
	Under the assumptions of \ref{theorem:pressure_convergence} along with Assumption \ref{assumption:conv} the following inequality on the pressure error holds.
	\begin{equation}
	\begin{aligned}
	 \beta_{m}|||e_{p}|||_{1,0} &\leq C\bigg\{\alpha C^{H^{1}}_{r}\sqrt{((1 + |||{\mathbb S}_r|||_{2}) (h^{2s} + \dt^{2})  + \sum_{i=r+1}^{N_{V}}\lambda_i + \sum_{i=r+1}^{N_{V}}\lambda_i \|\nabla \varphi_{i}\|^{2}}
	 \\
	 &+ \sqrt{h^{2k} + \dt^{2} + \sum_{i=m+1}^{N_{P}}\sigma_{i}}\bigg\}.
	\end{aligned}
	\end{equation}
	\begin{proof}
		Using the regularity condition from Assumption \ref{assumption:regularity}, applying the estimates from Theorem \ref{theorem:velocity_err} and Assumption \ref{assumption:conv} to the inequality from Theorem \ref{theorem:pressure_convergence} the result follows.
		\end{proof}
\end{corollary}


\section{Numerical Experiments}\label{sec:numerical_experiments}
In this section, we perform a numerical investigation of the MER formulation \eqref{eqn:posteriori-momentum} and the PPE \eqref{eqn:D-PPE}.  To carry out the numerical experiments, we utilize the FEniCS software suite \cite{LNW12}.

\subsection{Problem Setting}
The problem setting is the same as that used in Section 6 of \cite{DILMS19}.  Letting $r_{1}=1$, $r_{2}=0.1$, $c_{1}=1/2$, and $c_{2}=0$; the domain is given by
\[
\Omega=\{(x,y):x^{2}+y^{2}\leq r_{1}^{2} \text{ and } (x-c_{1})^{2}%
+(y-c_{2})^{2}\geq r_{2}^{2}\}.
\]
This represents a disk with a smaller off-center disc inside (see Fig. \ref{fig:mesh}). 

\begin{figure}[!ht]
	\centering
	\includegraphics[width = .25\linewidth]{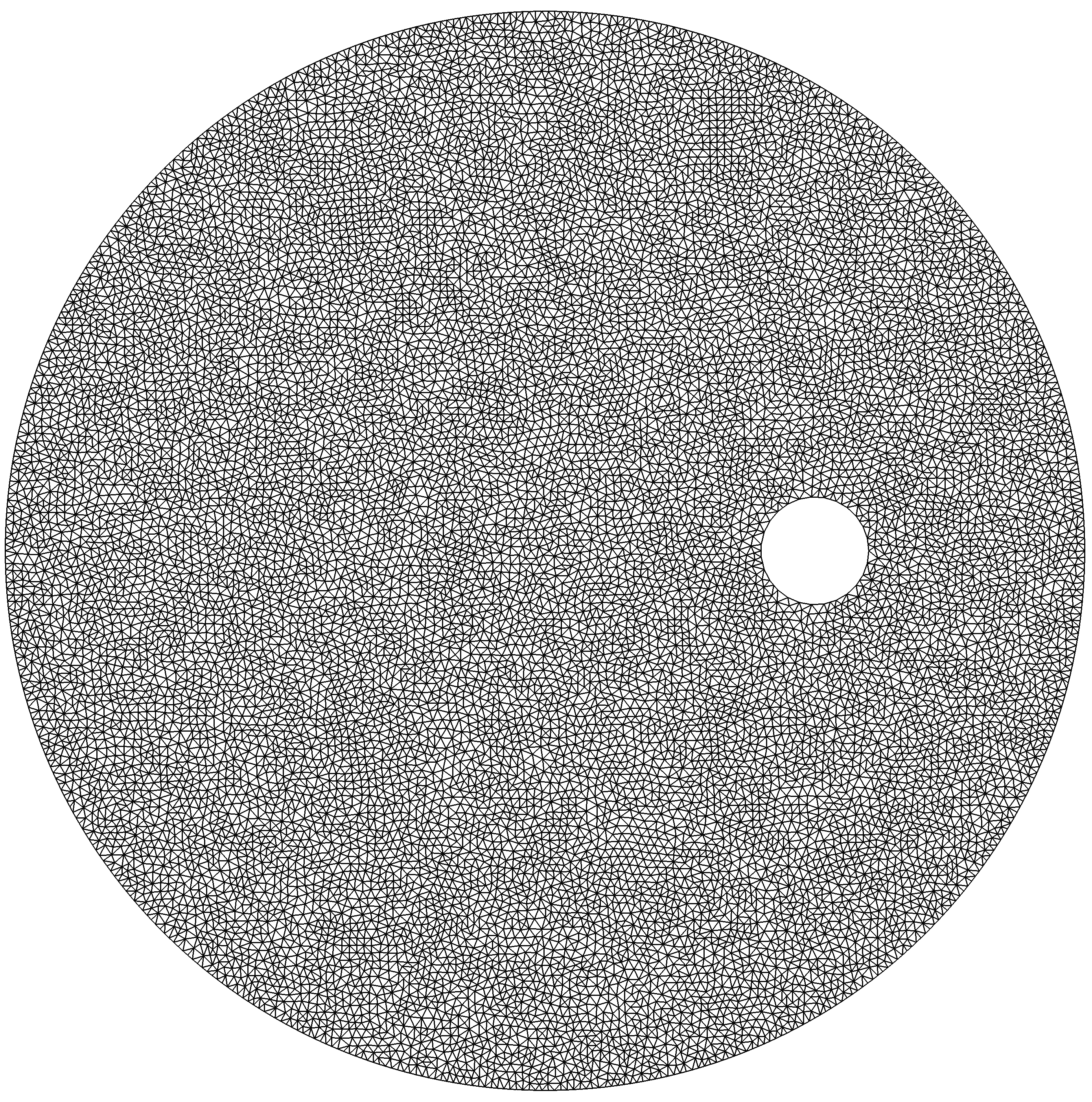}
	\caption{Spatial mesh for the finite element approximation.}
	\label{fig:mesh}
\end{figure}

%
The viscosity is $\nu = \frac{1}{100}$ and the counterclockwise rotational body force is given by
\begin{equation}\notag
f(x) = (-4y(1-x^2-y^2),4x(1-x^2-y^2)).
\end{equation}
No-slip boundary conditions are imposed on both cylinders. Because of the fact that $f = 0$ at the outer circle, most of the complex structures occur from the interaction of the flow with the inner cylinder. Specifically, the inner cylinder causes Von K\'{a}rm\'{a}n vortex street to develop, which then rotates and reinteracts with the inner cylinder. 

For the offline calculation, the snapshots are calculated via the $P^{2}-P^{1}$ Taylor-Hood backward Euler discretization \eqref{eqn:BE_FEM}. The flow is initialized at rest with $u_h^0 \equiv 0$. The velocity space $X_h$ and pressure space $Q_h$ have 114,792 and 14,474 degrees of freedom, respectively. We take $\Delta t=2.5e-4$ and collect velocity and pressure snapshots at every time step in the interval $[12,16]$.  The first fifty singular values for the velocity and pressure are shown in Fig. \ref{singular_value_plots}.

\begin{figure}
	\centering
	\begin{subfigure}[b]{0.48\linewidth}
		\includegraphics[width=\linewidth]{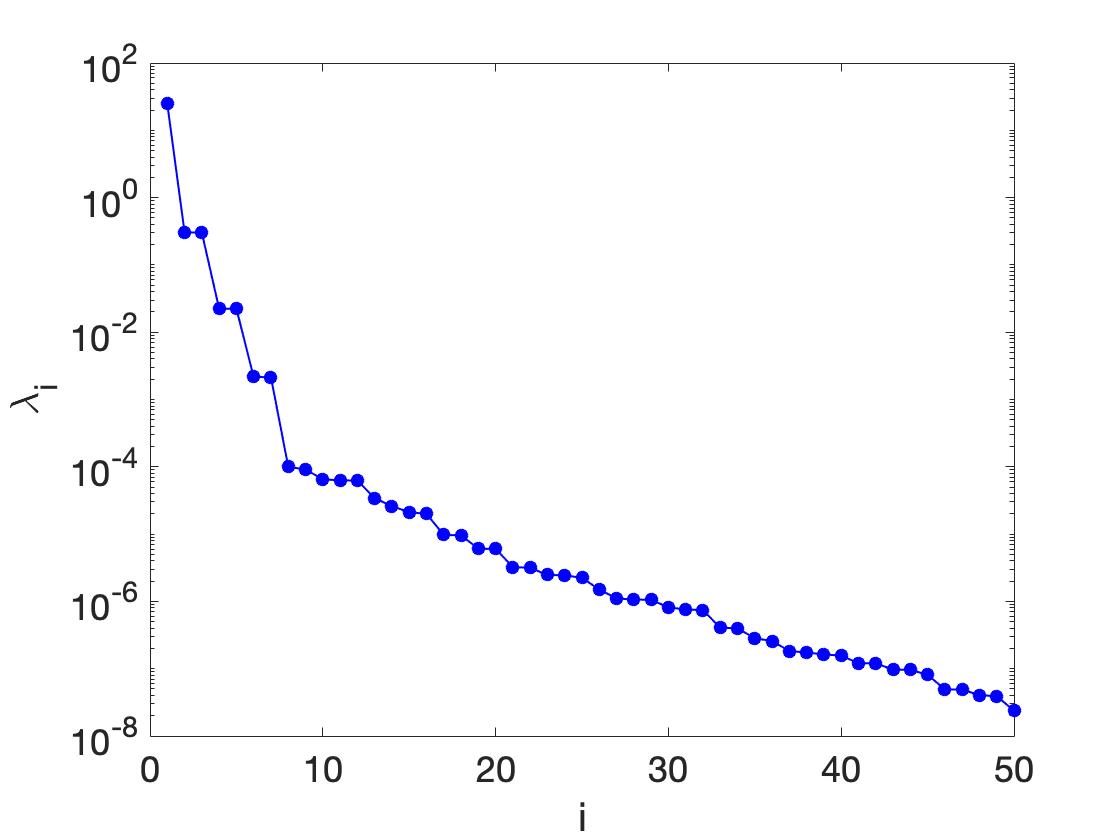}
	\end{subfigure}
	\begin{subfigure}[b]{0.48\linewidth}
		\includegraphics[width=\linewidth]{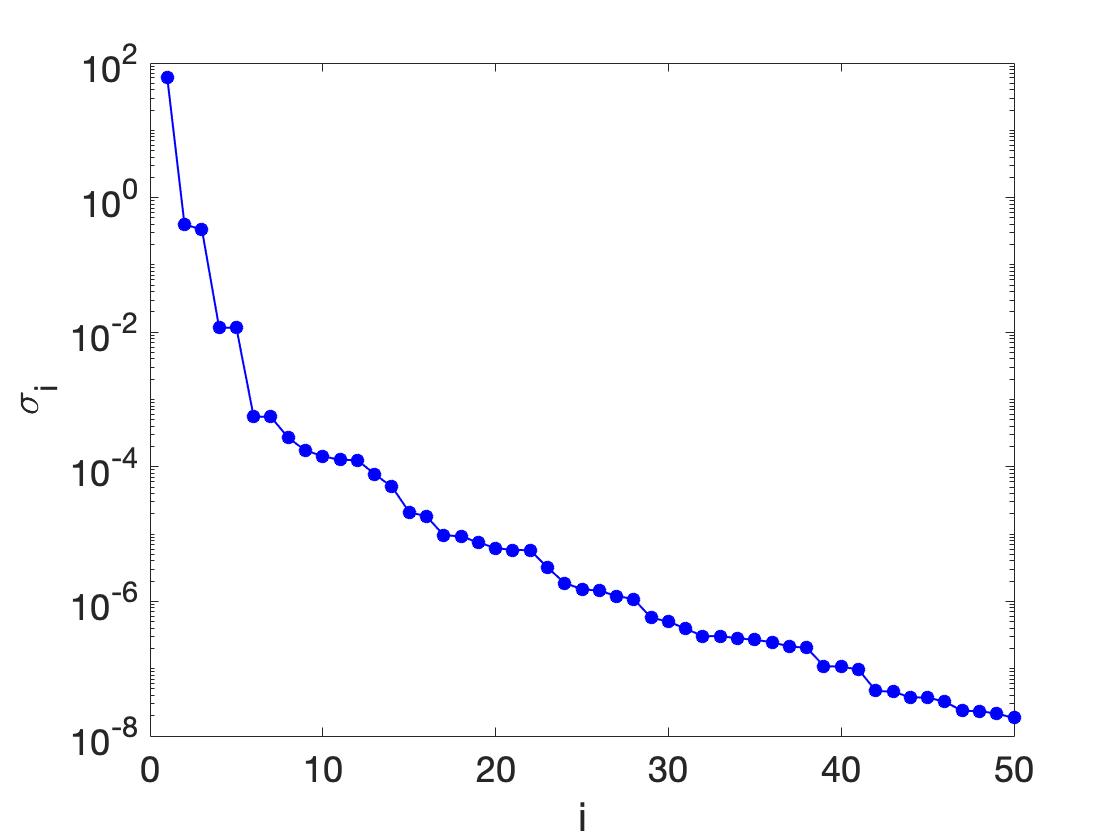}
	\end{subfigure}
	\caption{The first 50 singular values for the velocity (left) and pressure (right) modes. }
	\label{singular_value_plots}
\end{figure}

The smaller cylinder exerts a force due to lift and a force due to drag on the flow. The drag force is in opposition to the counterclockwise rotation, and the force due to lift is perpendicular to the rotation, in this case chosen to be inward. We calculate the lift and drag using the volume integral approach from \cite{J04}.




%

\subsection{MER Convergence Test}

In this section, we numerically verify the convergence rates for the pressure determined by the MER formulation with respect to the ROM projection errors established in Theorem \ref{theorem:pressure_convergence}. We measure the $\ell^{1}L^{2}$ error between the ROM solution $p_{m}$ and the offline solution $p_{h}$ for varying values of $r$ and $m$. The same stepsize $\Delta t=2.5e-4$ used in the offline stage is used in the calculation of the ROM solution.  

Corollary \ref{corollary:convergence} shows that the pressure error bound depends on $h$, $\Delta t$, $|||{\mathbb S}_r|||_{2}$, $\beta_{m}$, $\alpha C_r^{H_1}$, and the ROM truncation errors $\Lambda_m = \sqrt{\sum_{i=m+1}^{N_{P}}\sigma_{i}}$ and \newline $\Lambda_r =\sqrt{\sum_{i=r+1}^{N_{V}}\lambda_{i} + \sum_{i=r+1}^{N_{V}}\lambda_{i}\| \nabla \varphi_i \|^{2}}$. Because we are comparing the MER solution, $p_{m}$, to the offline solution, $p_{h}$, with the same underlying spatial and time discretization, the contribution to the error from terms involving $h$ and $\Delta t$ will be negligible.   Therefore, we examine the convergence of the pressure with respect to the terms $\beta_{m}$, $\alpha C_r^{H_1}$, $\Lambda_m$, and $\Lambda_r$.

First, we examine the convergence with respect to $\Lambda_m$. Setting $r = 50$, $\Lambda_r$ becomes negligible, therefore isolating the dependence of the pressure recovery error on $\Lambda_m$. In Fig. \ref{alphbeta_1}, we see that the inf-sup constant $\beta_{m}$ and term $\alpha C^{H_{1}}_{r}$ remain well behaved for $m=1$ to $m=50$. Thus, Corollary \ref{corollary:convergence} predicts the following convergence of $\vvvert e_p^{MER} \vvvert_{1,0}$ with respect to $\lambda_m$: 
\begin{equation}
\vvvert e_p^{MER} \vvvert_{1,0}  = \mathcal{O}(\Lambda_m).
\end{equation}
We list the error $\vvvert e_p^{MER} \vvvert_{1,0}$ for increasing $m$ in Table \ref{table:MER_lambda_m}. In addition, the corresponding power law regression is given in Fig. \ref{mer-conv-lambdam}. This regression agrees with our theoretical estimate, yielding:
\begin{equation}
\vvvert e_p^{MER} \vvvert_{1,0}  = \mathcal{O}((\Lambda_m)^{.993}).
\end{equation}

\begin{figure}
	\centering
	\begin{subfigure}{0.4\linewidth}
		\includegraphics[width=\linewidth]{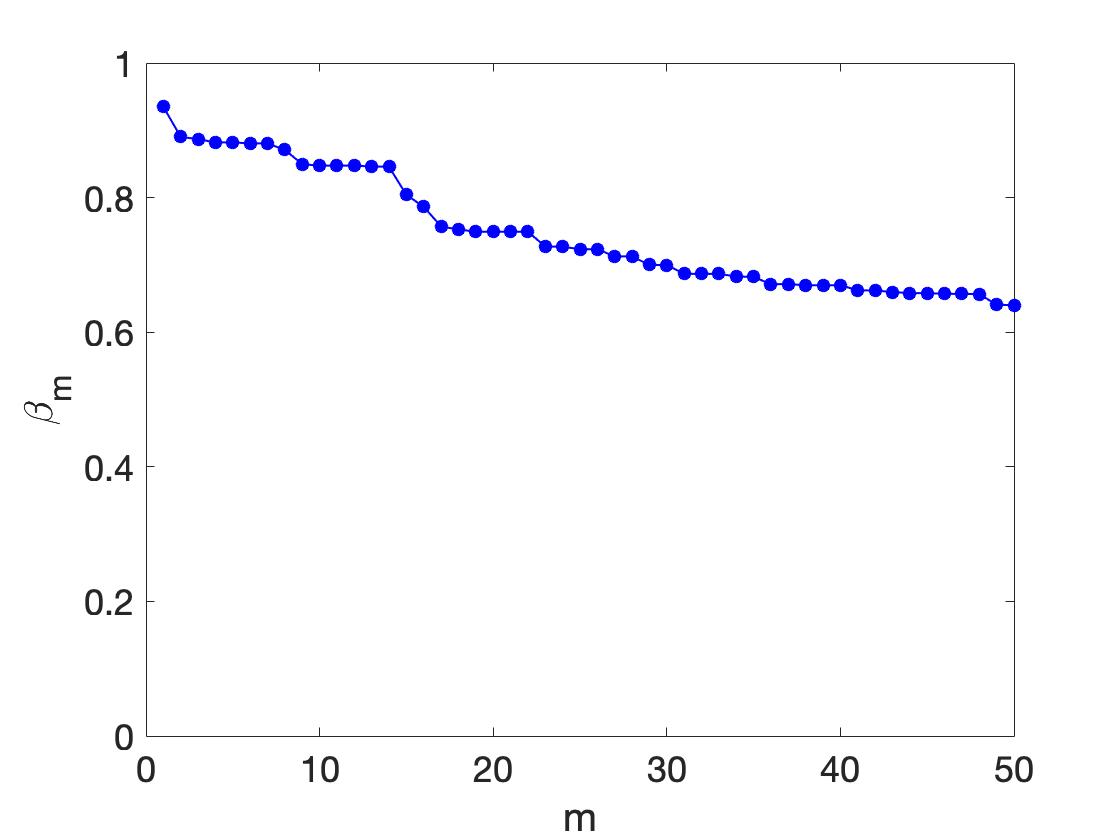}
	\end{subfigure}
	\begin{subfigure}{0.4\linewidth}
		\includegraphics[width=\linewidth]{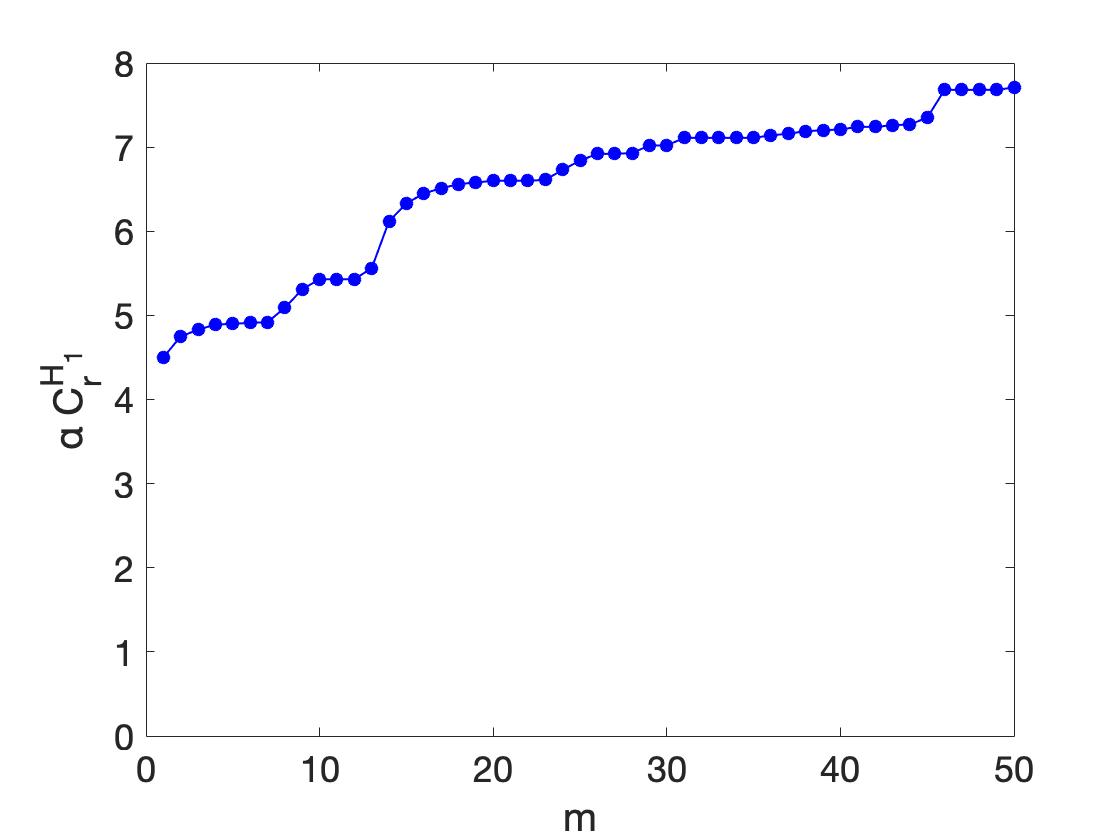}
	\end{subfigure}
	\caption{Value of the inf-sup constant, $\beta_m$, (left) and $\alpha C^{H_{1}}_{r}$ (right).}
	\label{alphbeta_1}
\end{figure}

\begin{center}
	\begin{table}
		\centering\renewcommand{\arraystretch}{1.3}
		\begin{tabular}{|c| c| c |} 
			\hline
			$m$ & $\vvvert e_p^{MER} \vvvert_{1,0}$ & $\Lambda_m$ \\
			\hline
			3 & 6.533e-01 & 1.596e-01 \\ 
			\hline
			6 & 1.594e-01 & 4.021e-02 \\
			\hline
			9 & 1.028e-01 & 2.495e-02  \\
			\hline
			12 & 5.762e-02 & 1.504e-02\\
			\hline
			15 & 3.494e-02 & 8.767e-03 \\ 
			\hline
			18 & 2.586e-02 & 6.293e-03 \\ 
			\hline
			21 & 1.928e-02 & 4.482e-03 \\ 
			\hline
			24 & 1.432e-02 & 3.039e-03 \\ 
			\hline
			27 & 1.002e-02 & 2.253e-03 \\ 
			\hline
			30 & 7.838e-03 & 1.709e-03 \\ 
			\hline
		\end{tabular}
		\caption{MER approximation errors for increasing $m$ values.}
		\label{table:MER_lambda_m}
	\end{table}
\end{center}

\begin{figure}
	\centering
	\begin{subfigure}{0.4\linewidth}
		\includegraphics[width=\linewidth]{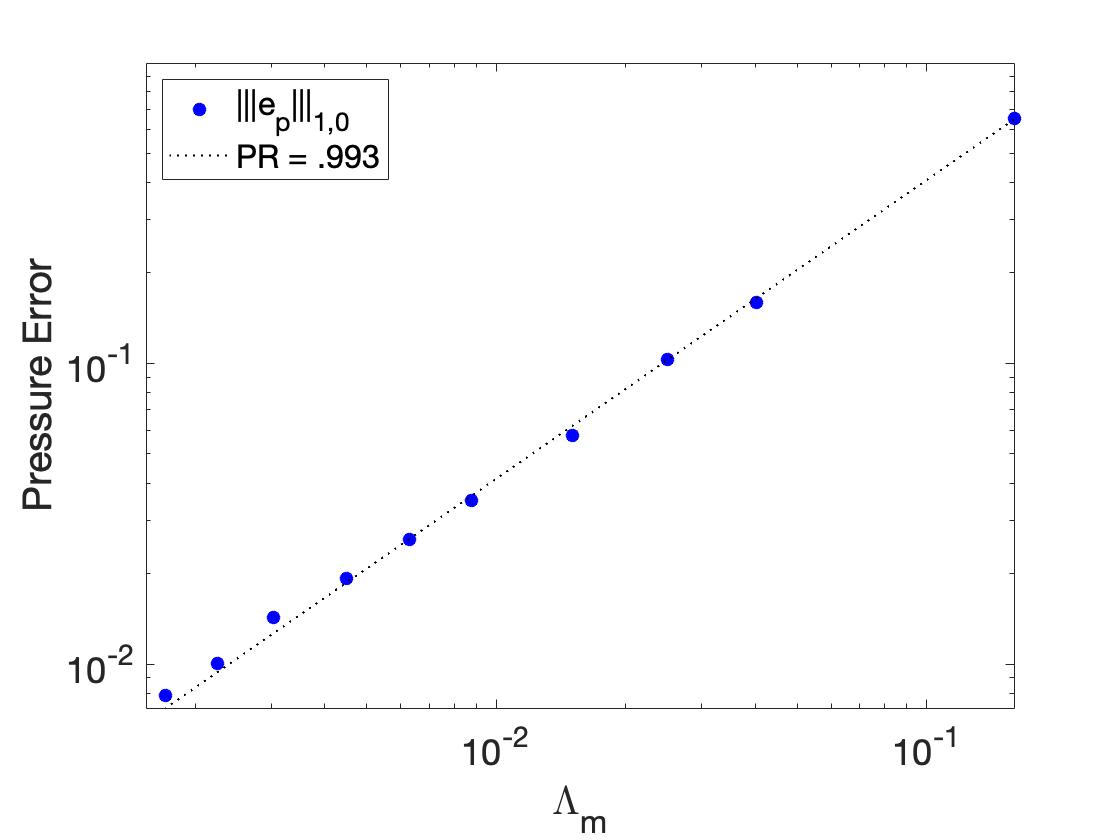}
	\end{subfigure}
	\caption{Power law regression of $\vvvert e_p^{MER} \vvvert_{1,0}$ with respect to $\Lambda_m$. }
	\label{mer-conv-lambdam}
\end{figure}

Next, we examine the convergence with respect to $\Lambda_r$. Setting $m = 50$, we isolate the relationship between pressure recovery error and $\Lambda_r$. For fixed $m,$ the inf-sup value stays constant with $\beta_m = .6402$. In Fig. \ref{alphbeta_2}, we show that $\alpha C^{H_{1}}_{r}$ grows slowly for $r=1$ to $r=50$. Corollary \ref{corollary:convergence} predicts the following convergence of $\vvvert e_p^{MER} \vvvert_{1,0}$ with respect to $\lambda_r$: 
\begin{equation}
\vvvert e_p^{MER} \vvvert_{1,0}  = \mathcal{O}(\Lambda_r).
\end{equation}
In Table \ref{table:MER_lambda_r}, we list the error $\vvvert e_p^{MER} \vvvert_{1,0}$ for increasing $r$. In addition, we give the corresponding power law regression in Fig. \ref{mer-conv-lambdar}. This agrees with our theoretical estimate, yielding:
\begin{equation}
\vvvert e_p^{MER} \vvvert_{1,0}  = \mathcal{O}((\Lambda_r)^{1.32}).
\end{equation}

\begin{figure}
	\centering
	\begin{subfigure}{0.4\linewidth}
		\includegraphics[width=\linewidth]{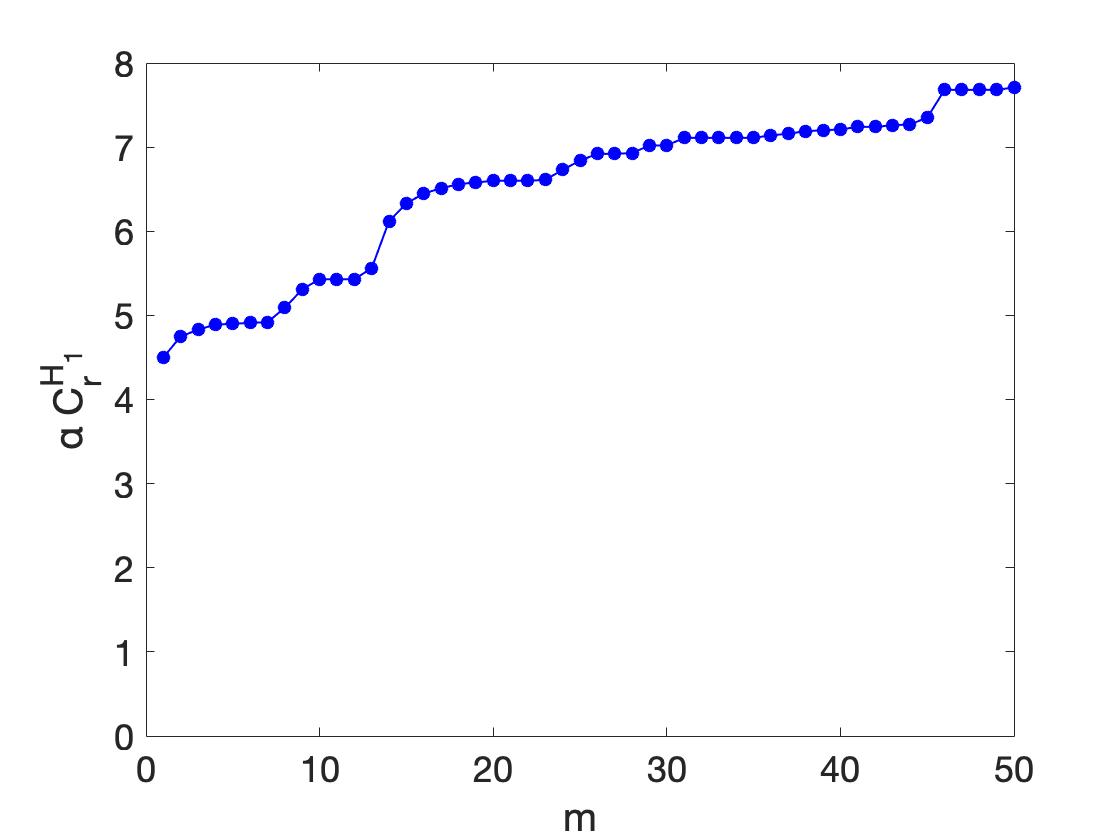}
	\end{subfigure}
	\caption{Value of $\alpha C^{H_{1}}_{r}$ with $m = 50$ and varying $r$.}
	\label{alphbeta_2}
\end{figure}

\begin{center}
	\begin{table}
		\centering\renewcommand{\arraystretch}{1.3}
		\begin{tabular}{|c| c| c |} 
			\hline
			$r$ & $\vvvert e_p^{MER} \vvvert_{1,0}$ & $\Lambda_r$ \\
			\hline
			10 & 2.500e-01& 1.922e+01  \\ 
			\hline
			15& 9.054e-02 & 1.062e+01  \\
			\hline
			20 & 4.407e-02& 8.225e+00   \\
			\hline
			25 & 2.704e-02 & 5.421e+00\\
			\hline
			30& 1.338e-02 & 2.963e+00  \\ 
			\hline
			35& 9.698e-03 & 2.652e+00  \\ 
			\hline
			40& 6.195e-03 & 2.582e+00  \\ 
			\hline
			45& 5.885e-03 & 1.302e+00  \\ 
			\hline
			50 & 3.442e-03& 1.051e+00  \\ 
			\hline
		\end{tabular}
		\caption{MER approximation errors for increasing $r$ values.}
		\label{table:MER_lambda_r}
	\end{table}
\end{center}

\begin{figure}
	\centering
	\begin{subfigure}{0.48\linewidth}
		\includegraphics[width=\linewidth]{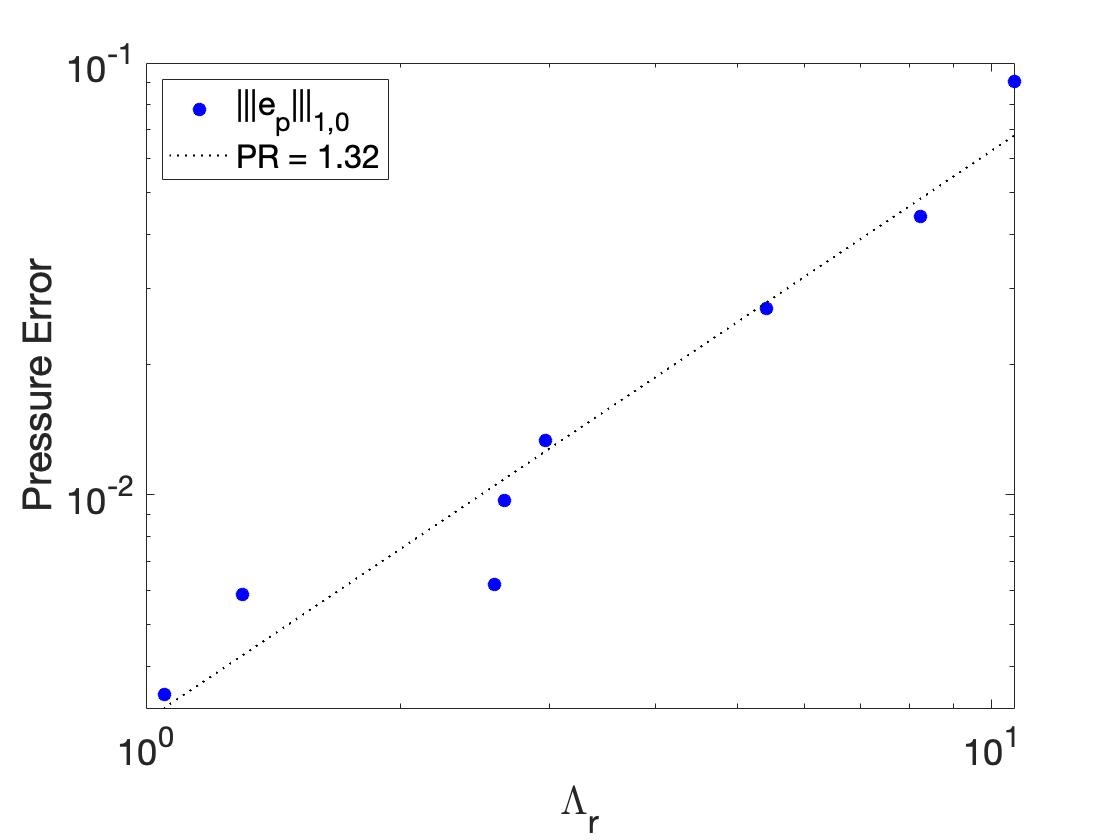}
	\end{subfigure}
	\caption{Power law regression of $\vvvert e_p^{MER} \vvvert_{1,0}$ with respect to $\Lambda_r$. }
	\label{mer-conv-lambdar}
\end{figure}

\subsection{Comparison between MER and PPE}
Lastly, we compare the performance of the PPE against the MER formulation for pressure recovery. To this end, we set $r=50$ while varying $m$ to examine the convergence of the two schemes as the size of the pressure basis increases.  For both approaches, we calculate the force due to lift, the force due to drag, and the errors $\vvvert e_p \vvvert_{1,0}$. 

Based on the discussion from Remark \ref{remark:ppe}, we expect there to be a consistency error present in the PPE pressure solution due to the Neumann boundary condition. In Table \ref{table:PPE-MER-comp}, we list the errors for the MER solution, $\vvvert e_p^{MER} \vvvert_{1,0}$ and the PPE solution for increasing values of $m$. While the MER solution error improves for increasing $m$, the error for the PPE stagnates. We can also see the error stagnation in the time evolution of the lift and drag error for $m=21$ in Fig. \ref{fig:lift_drag_errors}. In Fig. \ref{Avg_err_plots}, we show the time-averaged pressure error for the PPE and MER methods with $m = 50$. We see that the error for the PPE approach is primarily located at the boundary of the smaller offset cylinder, where the average error for the MER is evenly distributed throughout the domain.

\begin{table}
	\centering\renewcommand{\arraystretch}{1.3}
	\begin{tabular}{|c| c| c |} 
		\hline
		$m$ & $\vvvert e_p^{MER} \vvvert_{1,0}$ & $\vvvert e_p^{PPE} \vvvert_{1,0}$ \\
		\hline
		3 & 6.533e-01 & 6.754e-01 \\ 
		\hline
		6 & 1.594e-01 & 2.530e-01 \\
		\hline
		9 & 1.028e-01 & 2.247e-01  \\
		\hline
		12 & 5.762e-02 & 2.090e-01\\
		\hline
		15 & 3.494e-02 & 1.819e-01 \\ 
		\hline
		18 & 2.586e-02 & 1.781e-01 \\ 
		\hline
		21 & 1.928e-02 & 1.738e-01 \\ 
		\hline
		24 & 1.432e-02 & 1.733e-01 \\ 
		\hline
		27 & 1.002e-02 & 1.773e-01 \\ 
		\hline
		30 & 7.838e-03 & 1.756e-01 \\ 
		\hline
	\end{tabular}
	\caption{Pressure error for MER and PPE approximations with $r = 50$ and varying $m$.}
	\label{table:PPE-MER-comp}
\end{table}

\begin{figure}
	\centering
	\begin{subfigure}[b]{0.4\linewidth}
		\includegraphics[width=\linewidth]{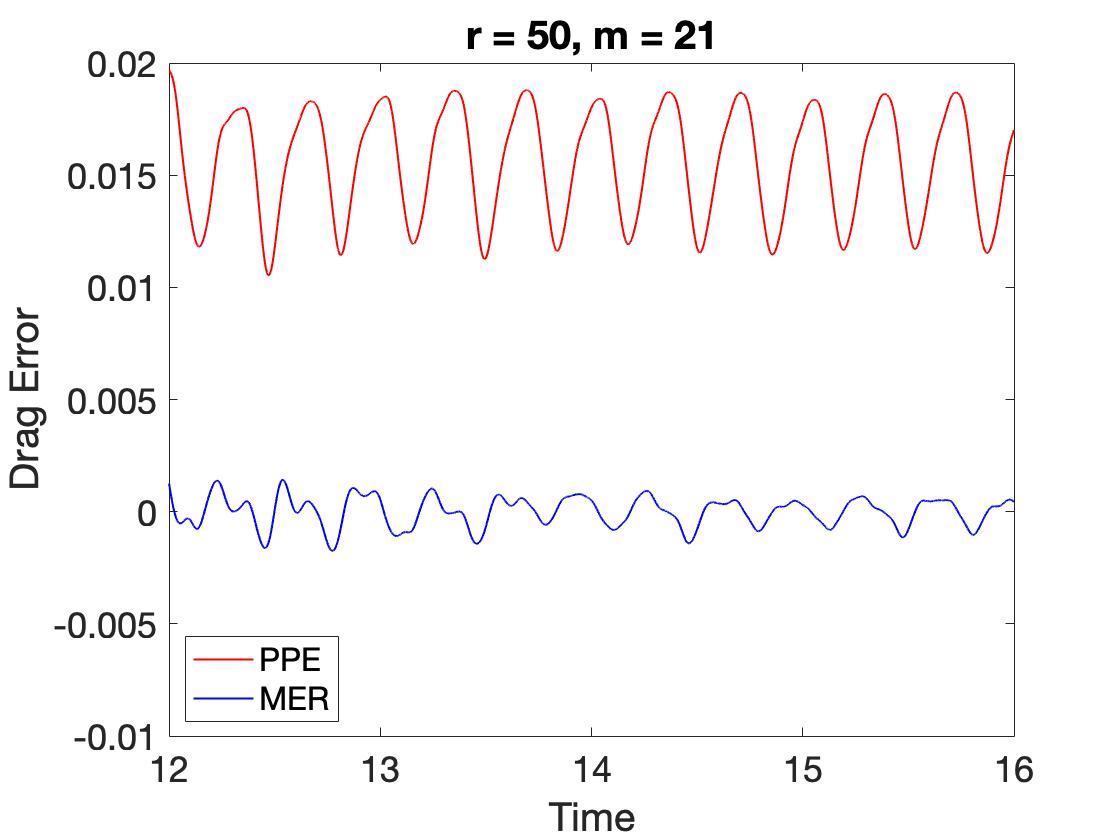}
	\end{subfigure}
	\begin{subfigure}[b]{0.4\linewidth}
		\includegraphics[width=\linewidth]{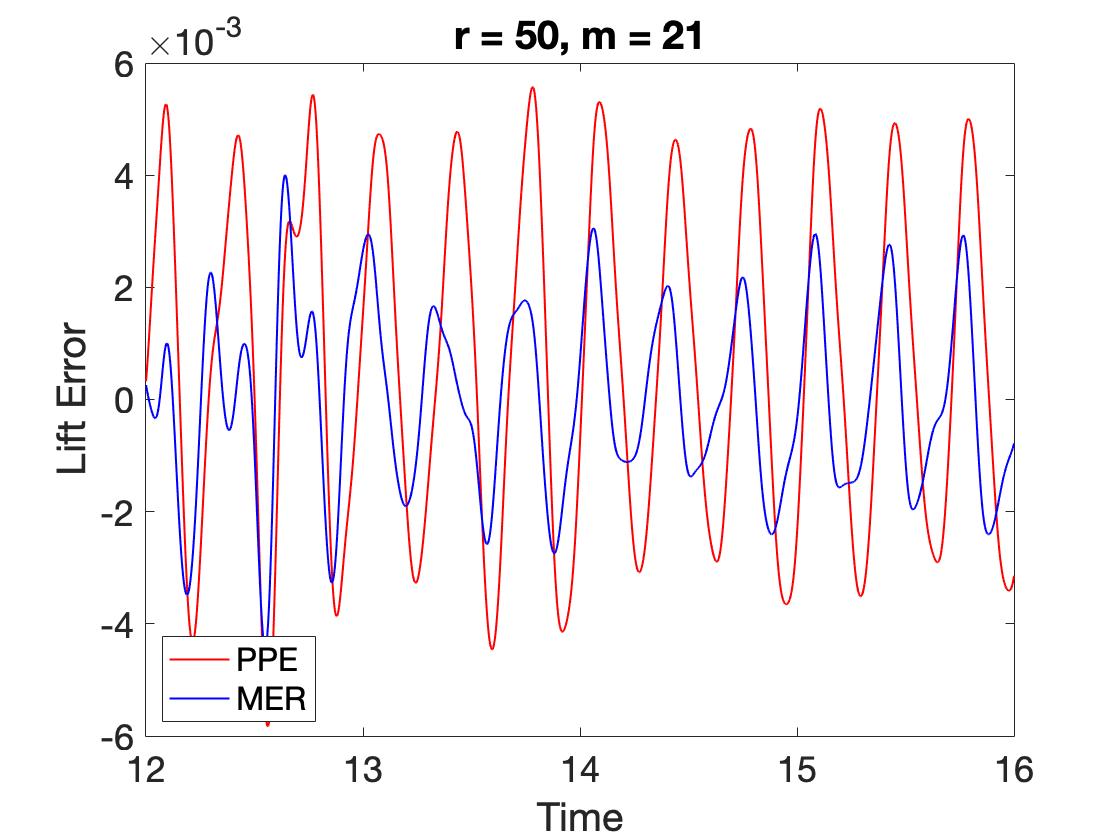}
	\end{subfigure}
	\caption{Time evolution of the drag (left) and lift (right) errors with $r = 50$ and $m =21$.}
	\label{fig:lift_drag_errors}
\end{figure}

\begin{figure}
	\centering
	\begin{subfigure}{0.48\linewidth}
		\includegraphics[width=\linewidth]{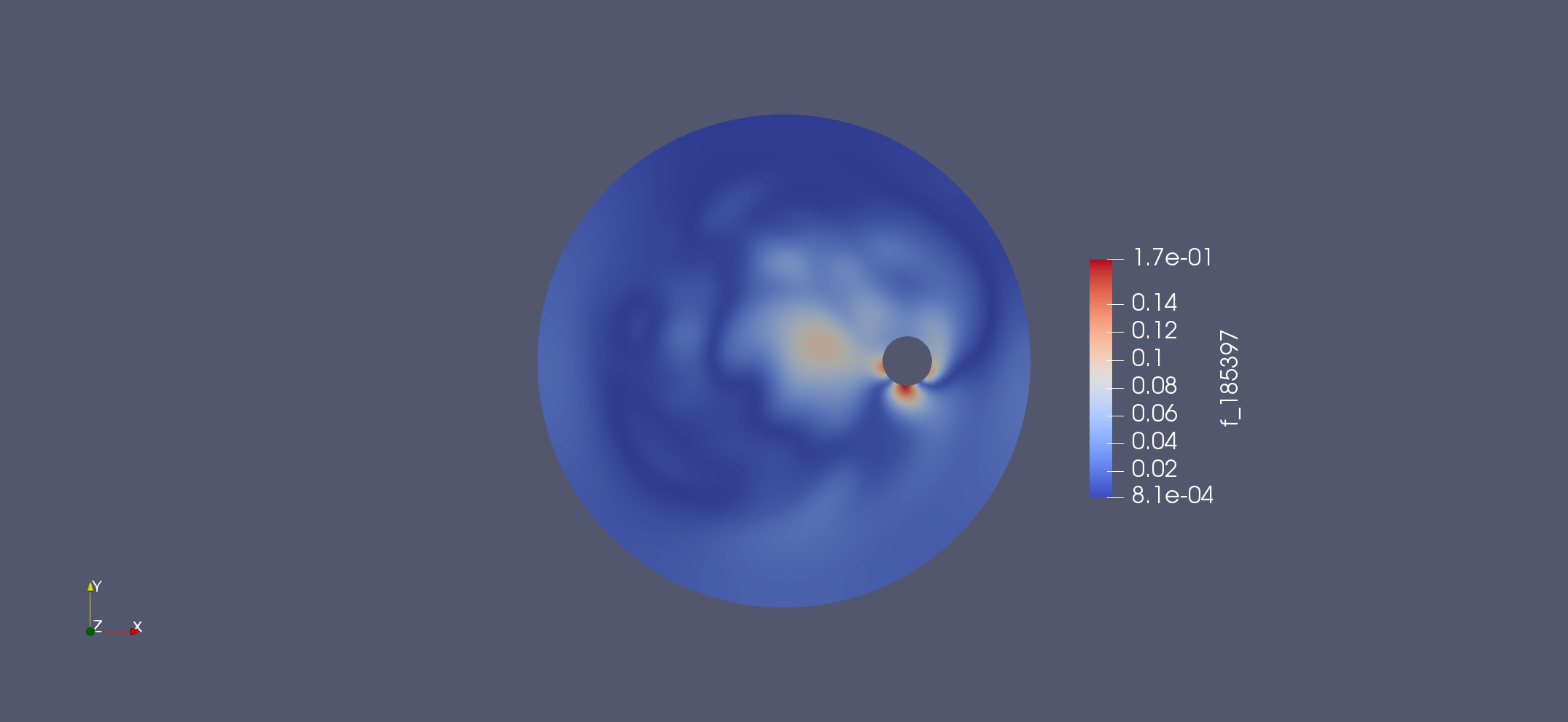}
	\end{subfigure}
	\begin{subfigure}{0.48\linewidth}
		\includegraphics[width=\linewidth]{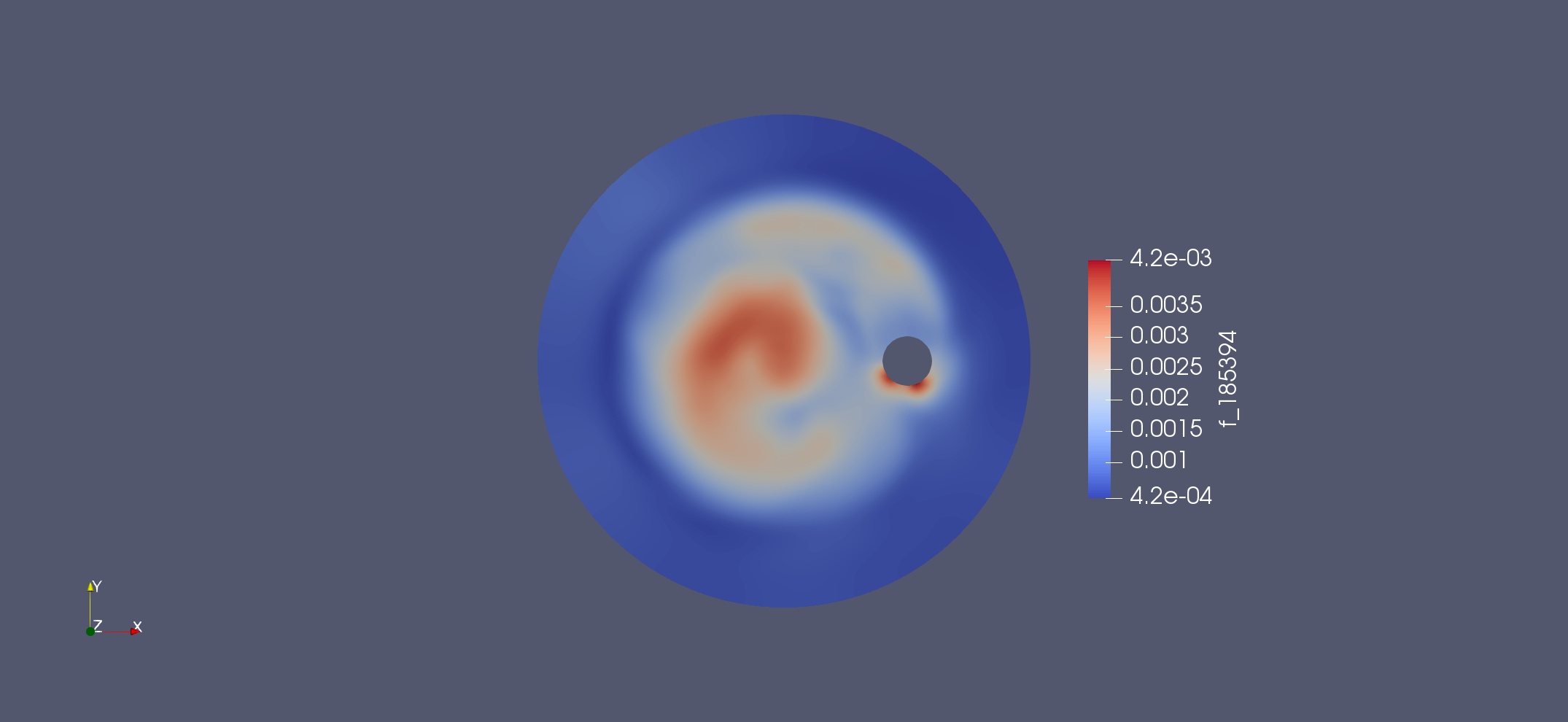}
	\end{subfigure}
	\caption{Time averaged error for the pressure solution recovered from the PPE (left) and MER (right) methods with r = m = 50.}
	\label{Avg_err_plots}
\end{figure}

\section{Conclusion}\label{sec:conclusions}
In this paper, we analyze the MER approach for recovering the pressure from a velocity-only ROM. We prove stability and convergence of the method and conduct numerical experiments illustrating the efficacy of this approach. Additionally, we perform a numerical comparison of the MER and PPE approach. We see that the Neumann boundary condition present in the PPE formulation leads to a loss of accuracy when a $C^{0}$ finite element space is used in the offline basis construction. 

In the future, we intend to pursue multiple research directions. First, we will conduct an analysis of the MER scheme for the time-dependent NSE with a parameterized domain. Second, we will investigate improving the supremizer stabilization algorithm by accounting for the computable constant, $\alpha C^{H_{1}}_{r}$, in constructing the supremizer space. Lastly, will examine whether the loss of accuracy in the PPE approach still occurs when other numerical scheme such as finite volume  or discontinuous Galerkin methods are used to collect solution snapshots for the POD basis construction.

\bibliographystyle{plain}
\bibliography{schneier}
\end{document}